%% file: verArXiV.tex
\documentclass[a4paper, 10pt]{amsproc}
\usepackage{amsmath, amsthm,amssymb,amsfonts}
\usepackage{tikz}
\usetikzlibrary{matrix}
\usepackage{derivative}
\usepackage{algorithm} 
\usepackage{algorithmic}  
\usepackage[linesnumbered,ruled,lined,algo2e]{algorithm2e} 
\usetikzlibrary{arrows}
\usepackage{caption}
\usepackage{graphicx}
\usepackage{array}
\usepackage{textcomp}
\usepackage{defs}
\usepackage{subcaption}
\usepackage{todonotes}
\usepackage{multicol}
\usepackage{mathtools}
\usepackage{chngcntr}
\usepackage{float}
\usepackage{booktabs}
\usepackage{array, multirow, url}

\title[Explicit feedback synthesis via quasi-interpolation]{Explicit feedback synthesis for nonlinear robust model predictive control driven by quasi-interpolation}
\author{Siddhartha Ganguly and Debasish Chatterjee}
\thanks{The current manuscript was submitted to IEEE Transactions on Automatic Control on 15 October, 2022. Siddhartha Ganguly is supported by the PMRF grant RSPMRF0262, from the Ministry of Human Resource Development, Govt.\ of India. The first author thanks Ravi Banavar for his encouraging comments.}
\thanks{\emph{Author Information}: \\
Address: Siddhartha Ganguly, and Debasish Chatterjee are with Systems and Control Engineering, Indian Institute of Technology Bombay, India.\\
Emails: \texttt{(sganguly,dchatter)@iitb.ac.in}.\\
Homepages: (SG): \url{https://sites.google.com/view/siddhartha-ganguly}; (DC): \url{https://www.sc.iitb.ac.in/~chatterjee}
}
\keywords{model predictive control, robust control, control policies, uniform approximation}

\begin{document}


\maketitle

\begin{abstract}
	We present QuIFS (Quasi-Interpolation driven Feedback Synthesis): an offline feedback synthesis algorithm for explicit nonlinear robust minmax model predictive control (MPC) problems with guaranteed quality of approximation. The underlying technique is driven by a particular type of grid-based quasi-interpolation scheme. The QuIFS algorithm departs drastically from conventional approximation algorithms that are employed in the MPC industry (in particular, it is neither based on multi-parametric programming tools nor does it involve kernel methods), and the essence of its point of departure is encoded in the following \embf{challenge-answer} approach: Given an error margin \(\eps>0\), compute in a single stroke a feasible feedback policy that is \emph{uniformly} $\varepsilon$-close to the optimal MPC feedback policy for a given nonlinear system subjected to constraints and bounded uncertainties. Closed-loop stability and recursive feasibility under the approximate feedback policy are also established. We provide a library of numerical examples to illustrate our results.
\end{abstract}

\input{verArXiVintro}

\input{verArXiVsetup}

\input{verArXiVmath}

\input{verArXiVresults}


\input{verArXiVlinear}

\input{verArXiVnum}

\input{verArXiVconcl}

\bibliographystyle{amsalpha}
\bibliography{refs}

\input{verArXiVapp}

\end{document}

%% file: verArXiVintro.tex
\newcommand{\paramh}{h}
\newcommand{\paramD}{\mathcal{D}}

\section{Introduction}
\label{sec:intro}

Model predictive control (MPC) is a model-based dynamic optimization method and it has evolved into one of the most practical and suitable control synthesis methodology over the years. MPC has found its way into several industries such as chemical, oil and gas production, electrical, finance, and robotics, apart from a host of others. It distinguishes itself in being perhaps the most versatile technique for incorporating constraints of a given problem directly at the synthesis stage, thereby directly adhering to the idea behind equipping machines with intelligence and achieving a clean technique of automation. The readers are pointed towards the survey articles \cite{ref:May-14, ref:May-16} and the detailed texts \cite{ref:RawMayDie-17, ref:GruPan-17} for a panoramic picture of the area. 

\subsection*{Background}

It is well-known \cite{ref:MAYNE-survey} that an MPC strategy provides a \emph{feedback} implicitly because the control action is dependent on the states at each discrete time instant \(t\). The task of \embf{explicit MPC} is to extract this implicit feedback and furnish the corresponding feedback mapping.\footnote{Of course, a \emph{mapping} can be defined when at each initial state there exists a unique solution of a given MPC problem; even otherwise, it is conceivably possible to appeal to the axiom of choice to define a map, and/or to the diverse array of selection theorems in the theory of set-valued maps to construct such a map with specific properties.}

The industry of \emph{explicit MPC} has a rich history, and we point the interested reader to the detailed survey article \cite{ref:AleBem-09} for a sweeping perspective of the area. The importance of the explicit method is underscored by the fact that the online computation of receding horizon control law at each \(t\) may be replaced by a function evaluation at each given state. This mechanism, at least in spirit, speeds up the computation of the MPC action by orders of magnitude, and primarily for this reason explicit MPC has found applications in several industrial plants; we refer the readers to \cite{ref:Morari_explicit, ref:kva:EMPC:maglev, ref:EMPC:Deepak:thesis} for more information. Most of the techniques in explicit MPC rely essentially on multiparametric programming \cite{ref:BemMorDuaPis-02, ref:Kva:jone:real_time, ref:kva:EMPC:region_less}, and while exact characterizations of optimal feedbacks are available for a wide class of systems, for numerical tractability reasons most of the applicable results are limited to the \emph{linear/affine} models. In this linear/affine regime, under mild hypotheses, the optimal implicit feedback turns out to be a piecewise affine mapping \cite{ref:BemMorDuaPis-02}. Several approaches to explicit MPC for nonlinear models have been developed, and ``approximation'' seems to be the driving force behind them; naturally, such efforts are accompanied by the key computational challenge of our times --- \textsl{the curse of dimensionality}, and that problem persists herein. Among the vast literature on the subject, we mention the following: A binary search tree and orthogonal decomposition-based algorithm to approximate the feedback function via piecewise affine approximations was established in \cite{ref:TJohan-aut-2004-NMPC} and its precursors. In \cite{ref:Canale2009-BookCh} a survey of set membership-based approximation methods for linear and nonlinear MPC problems was provided. Offline approximation of possibly discontinuous predictive laws was studied in \cite{ref:Pin-parsini-2012-IJC-di}. A multiresolution wavelet-based approximation method was introduced for both linear \cite{ref:summers-multires} and nonlinear MPC \cite{Raimondo2012} with guaranteed stability and feasibility of the feedback system; these contributions are perhaps closest to our approach although the estimates provided herein are \emph{uniform} and rigorous.

In the specific context of robust minmax MPC, offline explicit MPC techniques reported in \cite{ref:bemporad2003min, ref:Gao:EMPC, ref:DePena_Explicit_minmax}, are based on a partitioning of the state space into critical regions. These techniques cater to classes of linear controlled dynamical systems with bounded uncertainties. Most of these algorithms may fail to generate explicit control laws in situations where the prediction horizon is large. The primary reason for this is the potentially exponential growth of the number of polytopic regions with the number of constraints. We provide one such example where the explicit MPC algorithm terminated unsuccessfully due to the presence of large number of the polytopic regions when (approximate) multi-parametric programming-based tools were employed.\footnote{In contrast, our technique QuIFS produced visibly better results in terms of closeness to the online receding horizon control trajectories and the approximation quality; see Example \eqref{exmp:lmpc_0} in \S\eqref{sec:num_exp}.}

In the context of moderate-dimensional systems there are several explicit MPC algorithms based on neural networks (NN). Early contributions from the NN-based approximation perspective can be found in \cite{ref:parsini1995}, and more recently, NN-based approximation algorithms were advanced for linear dynamical systems in \cite{ref:karg:EMPC, ref:Kvasnica:EMPC, ref:chen2018approximating} but without guarantees of robust explicit approximation and closed-loop stability. In \cite{ref:Mesbah:EMPC} a projection operator-based and NN-enabled explicit MPC algorithm was proposed for linear systems and closed-loop guarantees were given without a control on the approximation error. Reformulating the closed-loop system in a diagonal differential inclusion, the authors in \cite{ref:nguyen2021robust} derived stability certificates for the NN-based MPC controller. In \cite{ref:hertneck2018learning} an NN-based approximation algorithm was established with soft guarantees of robust closed-loop stability; \cite{ref:rose2023learning} employed a Gaussian process (GP) framework to approximate the feedback law with soft guarantees of robust closed-loop stability and feasibility. Both of these preceding works provide probabilistic guarantees of approximation as opposed to uniform ones. In contrast to these algorithms, our technique provides one-shot uniform approximation guarantees along with stability and recursive feasibility guarantees; these guarantees are robust, and do \emph{not} involve probabilistic (soft) bounds. We also draw attention to \cite{ref:GanGupCha-23} where a NN-enabled explicit MPC scheme was advanced for linear systems with guarantees of preassigned uniform error (as distinguished from soft guarantees), stability and feasibility. We refer the readers to \cite{ref:weinan2022empowering} and the references therein for a recent survey on machine learning-enabled methods for MPC. We highlight the main features and our contributions next.
\subsection*{Our contributions}

In the article, we establish a novel grid-based technique --- QuIFS (Quasi-Interpolation driven Feedback Synthesis), for constructing explicit feedback maps based on quasi-interpolation. It is important to point out several aspects of our results at this stage, and we shall adopt a comparative rhetoric for listing the features of our technique for the ease of delineating our contributions:

\begin{enumerate}[label=(\Alph*), leftmargin=*, widest=I, align=left]
\item \label{prop:applicability_type} QuIFS is driven by a special type of interpolation on a uniform cardinal grid and provides strict guarantees of uniform approximation on potentially \emph{unbounded} sets under mild hypotheses. While we do not furnish the optimal feedback (we do \emph{not} solve the Bellman equation/recursion), we demonstrate that the difference between the optimal feedback and the approximating map can be made smaller than any \emph{pre-assigned} error tolerance \emph{in the uniform norm over the set of admissible states}. No other technique is, at the moment, capable of providing such strong guarantees.

\item \label{prop:applicability} 
The QuIFS algorithm applies to \emph{nonlinear systems} and \emph{non-convex cost functions} whenever the underlying optimal control problem admits a unique solution. It relies on coarse properties of the optimal feedback such as Lipschitz continuity, etc., rather than more detailed local structural properties; information concerning such coarse properties may be distilled directly from the problem data. This is a crucial point of departure from multiparametric techniques and deserves to be underlined. Of course, the optimal feedback is piecewise affine in the linear/affine setting under appropriate hypotheses; this important observation is now classical and follows from the central results of multiparametric programming in this context. In contrast, nonlinear models with non-convex cost functions may lead to a high degree of structural complexity, over and beyond the piecewise affine regime, of the optimal feedback, making it extremely difficult to find appropriate parametrization of such feedback maps.\footnote{In view of the current state of affairs of numerical analysis, parametrizing the optimal feedback (e.g., as along the lines of the Ritz method) does not appear to be a promising direction.} Ours being an interpolation-driven technique, an approach via multiparametric programming turns out to be unnecessary in our setting; merely the ability to compute solutions to finite-horizon optimal control problems at each point of the feasible set is sufficient.
	\item \label{prop:complexity} The complexity of the offline computations associated with QuIFS, \emph{as it stands today}, is exponential in the number of states because the technique relies on a uniform grid. Recall \cite[\S4.4]{ref:BemMorDuaPis-02} that the complexity of standard explicit MPC for linear/affine models scales exponentially with the number of constraints in the worst case. For us, however, the complexity scales exponentially \emph{only} with the state-dimension, and the number of constraints plays \emph{no} role.
	\item \label{prop:waterbed} QuIFS provides approximation error guarantees measured with respect to the \emph{uniform metric}. Traditional function approximations based on sampling techniques are broadly classifiable into parametric or non-parametric types. In the former case, expansion of the optimal feedback in terms of, e.g., Ritz bases or dictionaries, and in the latter case, expressing the unknown function as, e.g., a member of some reproducing kernel Hilbert space, are both rooted in Hilbert space methods. There appears to be no reasonable mechanism to prevent the Gibbs's phenomenon (waterbed effect) that typically leads to potentially wild and large-amplitude oscillations at the boundary of the feasible domains, and such approaches do \emph{not}, in general, provide \emph{uniform approximation error guarantees}. Since the uniform error metric is employed herein and is indeed necessary to ensure recursive feasibility (see \S\ref{subsec:BLMPC_to_ARMPC} for a discussion)\footnote{Probabilistic guarantees by themselves \emph{cannot} ensure recursive feasibility.}, our approach provides perhaps the closest approximation of the true optimal feedback available today, and this is achieved \emph{without} solving the associated Bellman equations/recursions.
	\item \label{prop:tools} Our chief technical tool --- a particular type of quasi-interpolation --- conforms to neither of the two types of interpolation-based approximation techniques mentioned in \ref{prop:waterbed}, and departs sharply from the typical approximation theoretic tools that ensure asymptotic convergence as the length of the sampling interval converges to \(0\). The upper bound on the uniform error is a function \(\psi_0(\cdot, \cdot, \cdot)\) of \emph{three} parameters --- the discretization interval \(\paramh > 0\), the shape parameter \(\paramD > 0\), and the truncation parameter \(\rzero>0\). For a \emph{prespecified uniform error margin} \(\eps > 0\), it is always possible to pick the triplet \((\paramh, \paramD, \rzero)\) such that \(\psi_0(\paramh, \paramD, \rzero) \le \eps\).\footnote{However, for a fixed $\paramD > 0$, we do \emph{not} have $\psi_0(\paramh, \paramD, \rzero) \xrightarrow[h\downarrow 0]{} 0$.} Consequently, the synthesis process is in principle \emph{one-shot}; iterative correction procedures involving steps such as first a selection of the discretization interval, followed by the verification of whether the ensuing error satisfies a given threshold, and retuning the discretization interval if the threshold is not attained, etc., are entirely unnecessary. For the same reason, there is no utility of the standard log-log plots of the error in our setting.
	\item \label{prop:error} It is possible to pick three parameters mentioned in \ref{prop:tools} a priori in order to ensure that the uniform error between the optimal feedback and the approximated one stays \emph{below the machine precision of floating point arithmetic}. Of course, the resulting computational burden may be difficult to achieve with low-fidelity hardware, but it is not so with the aid of sufficiently rich computational resources.
\end{enumerate}
\vspace{-1mm}

\subsection*{Organization}

This article unfolds as follows. The primary problem that we address in this article is formulated in \S \eqref{sec:prob_and_main_result}. The mathematical background on the quasi-interpolation engine we employ here is presented in \S\eqref{appendix_app_app_theory}. The main contributions of this article are presented in \S\eqref{sec:main_results}. Our main result consists of two parts: the Lipschtiz extension and interpolation algorithm and the closed-loop stability results under the proposed approximation scheme. In \S \eqref{sec:linear_case} we focus on the linear MPC regime and develop a specialized version of the more general theory presented in \S\eqref{sec:prob_and_main_result} and \S\eqref{sec:main_results}. Four numerical examples are presented in \S \eqref{sec:num_exp} to illustrate the effectiveness of the proposed algorithm. The Appendix \S \eqref{sec:appenb} contains preliminaries on input-to-state stability results.

\vspace{-2mm}

\subsection*{Notation}
Our time indices are denoted by \([m;n] \Let \left\{m,m+1,\ldots,n\right\}\), where \(m,n\in \N \Let \aset[]{1,2,\ldots}\) and \(m \le n\). We let \(\Nz\Let \Nzr\) denote the set of natural numbers and \(\Z\) denote the integers. The vector space \(\R[d]\) is assumed to be equipped with standard inner product \(\inprod{v}{v'}\Let \sum_{j=1}^d v_j v'_j\) for every \(v,v' \in \R[d]\). Let \(X\) be an arbitrary subset of \(\Rbb^d\); by \(\intr X\) we denote the interior of \(X\), and \(\partial X\) denotes the boundary of \(X\). For any \(A, B \subset \Rbb^d\) we define \(A \oplus B \Let \aset[]{ a+b \suchthat a \in A,\,b\in B}\). Given a set \(\mathcal{A} \subset \Rbb^d\) the point-to-set distance from any point \(p_0 \in \Rbb^d\) to \(\mathcal{A}\) is denoted by \(d\bigl(p_0,\mathcal{A}\bigr) \Let \inf_{p_1 \in \mathcal{A}}\norm{p_0-p_1}\). We denote the uniform function norm via the notation \(\unifnorm{\cdot}\); more precisely, for a real-valued bounded function \(f(\cdot)\) defined on a set \(S\), it is given by \(\unifnorm{f(\cdot)} \Let \sup_{x \in S}|f(x)|\). For vectors residing in some finite dimensional vector space \(\Rbb^d\) we employ the notation \(\norm{\cdot}_{\infty}\) to denote the usual uniform vector norm. For \(\ell \in \N\) and for \(p \in \lcrc{1}{+\infty}\), the \(\ell\)-dimensional closed ball centered at \(x\) and of radius \(\eta>0\) with respect to the \(p\)-norm (vector) is denoted by \(\Ball^{\ell}_{p}[x, \eta]\).  For \(X\) and \(Y\) open subsets of Euclidean spaces, the set of \(r\)-times continuously differentiable functions from \(X\) to \(Y\) is denoted by \(\mathcal{C}^r (X;Y).\) The Schwartz space of rapidly decaying \(\Rbb\)-valued functions \cite[Chapter 2]{ref:mazyabook} on \(\Rbb^d\) is denoted by \(\mathcal{S}(\Rbb^d)\).

%% file: verArXiVsetup.tex
\section{Problem formulation}\label{sec:prob_and_main_result}

The fundamental object of interest in this article is a discrete-time autonomous (possibly nonlinear) control system given by the recursion 
\begin{equation}
    \label{eq:system}
    \st_{t+1} = \field(\st_t,\ut_t,\dist_t),\quad \st_0\text{ given},\quad t \in \Nz,
\end{equation}
with the following data:
\begin{enumerate}[label=\textup{(1-\alph*)}, leftmargin=*, widest=b, align=left]
	\item \label{eq:system-st-con-dist}\(\st_{t} \in \Rbb^{d}\), \(\con_t \in \Rbb^{\dimcon}\), and \(\dist_t \in \Rbb^p\) are, respectively, the vectors of the states, the control actions, and the uncertainty elements at time \(t\);
	\item  \label{eq:system-dynamics} the `vector field' \( \Rbb^d \times \Rbb^{\dimcon} \times \Rbb^p \ni (\dummyx,\dummyu,\dummyw ) \mapsto \field( \dummyx,\dummyu,\dummyw ) \in \Rbb^d\) is continuous and \(\field(0,0,0)=0\);
	\item \label{eq:constraint-sets-st-ut} the system \eqref{eq:system} is subjected to the state and control action constraints 
		\[
			\st_t \in \Mbb \text{ and } \ut_t \in \Ubb \quad\text{for all }t \in \Nz,
		\]
		where \(\Mbb\) is a nonempty closed subset of \(\Rbb^d\), \(\Ubb\) is a nonempty compact subset of \(\Rbb^{\dimcon}\), each containing the respective origin in its interior; 
	\item \label{eq:constraint-sets-wt} the uncertainty elements in the system, captured by \(\dist_t\) at each time \(t\), is assumed to be bounded, i.e., 
		\[
			\dist_t \in \Wbb \quad \text{for all }t \in \Nz,
		\]
		where \(\Wbb\subset \Rbb^p\) is a compact set containing the element \(0 \in \Rbb^p\) in its interior.
\end{enumerate}
As key ingredients of the MPC strategy, we assume that the following data are given to us:
\begin{enumerate}[label=\textup{(1-\alph*)}, leftmargin=*, widest=b, align=left, start=5]
\item \label{eq:system:horizon} a time horizon \(\horizon \in \N\);
\item \label{eq:system:costs} a \emph{cost-per-stage} function
		\(
			\Rbb^d \times \Rbb^{\dimcon} \ni (\dummyx,\dummyu) \mapsto \cost(\dummyx,\dummyu) \in \lcro{0}{+\infty},
		\)
		and a \emph{final-stage cost} function
		\(
			\Rbb^d \ni \xi \mapsto \fcost(\xi)\in \lcro{0}{+\infty},
		\)
		satisfying \(\cost(0,0)=0\) and \(\fcost(0)=0\);
\item \label{eq:system:terminal_set} a specified terminal set \(\admfinst\), which is compact and contains the origin in its interior, and, 
\item \label{eq:system:policies} a class of admissible control policies \(\policies\) consisting of a sequence \(\policy\Let (\policy_t)_{t=0}^{\horizon-1}\) of measurable maps such that \(\pi_i: \Mbb \lra \Ubb\) for each \(i\) and we write \(\dummyu_t = \pi_t(\dummyx_t)\).
\end{enumerate}
\vspace{1mm}
Given the preceding ingredients, the \embf{baseline robust optimal control problem} underlying the MPC strategy for the system \eqref{eq:system} and its accompanying data \ref{eq:system-st-con-dist}-\ref{eq:system:policies} is given by
\begin{equation}
	\label{eq:baseline robust MPC}
	\begin{aligned}
		& \inf_{\policy(\cdot)} \sup_{W} && \sum_{t=0}^{\horizon-1} \cost(\dummyx_t, \dummyu_t) + \fcost(\dummyx_{\horizon})\\
		& \sbjto && \begin{cases}
			\dummyx_{t+1} = \field(\dummyx_t, \dummyu_t, \dummyw_t),\,\dummyx_0 = \xz,\\
		\dummyx_t\in\Mbb,\,\dummyx_\horizon\in\admfinst, \text{ and }\dummyu_t\in\admact\\
			\quad\text{for all }(\dummyw_t,t) \in\Wbb \times \timestamp{0}{\horizon-1},\\
			\dummyu_t = \policy_t(\dummyx_t),\,
			\policy(\cdot)\in\policies, \\ W \Let (\dummyw_0, \ldots, \dummyw_{\horizon-1}).
		\end{cases}
	\end{aligned}
\end{equation}
The `measured state' \(\xz\in\Mbb\) enters the minmax problem \eqref{eq:baseline robust MPC} as a parameter. A \emph{solution} to \eqref{eq:baseline robust MPC} is an optimal policy
\[
	\opt{\policy} (\cdot) \Let (\opt\policy_t(\cdot))_{t=0}^{\horizon-1}
\]
and by construction, it respects the state and control action constraints irrespective of the admissible uncertainties. It is well-known that for optimal control problems in the presence of uncertainties, control policies are the correct objects to be optimized (see the discussions in \cite{ref:May-16}), accordingly, our formulation of the baseline robust optimal control problem \eqref{eq:baseline robust MPC} underlying the MPC problem features the outer minimization over a class of policies. Observe that since \(\xz\) is a parameter in \eqref{eq:baseline robust MPC}, if the solution \(\opt\policy(\cdot)\) of \eqref{eq:baseline robust MPC} is unique, then the optimal policy \(\opt\policy(\cdot)\) is a \emph{mapping} of \(\xz\). In particular, if \(X_{\horizon}\) is the set of feasible initial states for which the MPC problem \eqref{eq:baseline robust MPC} admits a solution, then \(\opt{\policy}_0:X_N\lra\admact\) is a \emph{feedback}.\footnote{In the absence of uniqueness, one gets a set-valued map instead of a feedback map.} The robust MPC algorithm proceeds by measuring the states \(\st_t\) of \eqref{eq:system} at time \(t\), setting \(\xz = \st_t\) in \eqref{eq:baseline robust MPC}, solving \eqref{eq:baseline robust MPC} to obtain an optimal policy \(\opt{\policy}(\cdot)\), setting \(\con_t \Let \opt{\policy}_0(\xz)\) in \eqref{eq:system}, incrementing time to \(t+1\), and repeating the preceding steps.

\subsection{From the baseline MPC to the approximation-ready MPC: motivation and formulation}\label{subsec:BLMPC_to_ARMPC}
In the context of explicit MPC, our primary focus is on synthesizing a tight \emph{approximation} of the first entry \(\opt{\policy}_0(\cdot)\) of the optimal policy. We measure tightness with respect to the \emph{uniform norm}, and the designer is permitted to specify the threshold of tightness, say \(\eps > 0\), before the synthesis procedure. Accordingly, if \(\pi_0^{\dagger}(\cdot)\) is an approximation of \(\opt{\policy}_0(\cdot)\), then we stipulate that 
\begin{align}
	\unifnorm{\opt{\policy}_0(\cdot) - \pi_0^{\dagger}(\cdot)} \le \eps.
\end{align}
In other words, we have \(\|\opt{\policy}_0(y)-\pi_0^{\dagger}(y)\| \le \eps\) for all \(y\in X_N\).

Such an approximation procedure (see \S\eqref{appendix_app_app_theory} for more details on the procedure and the corresponding estimates) generates an error that enters the system in the form of uncertainty in the control actions. Naturally, one must \emph{accommodate these uncertainties at the design stage} in order to ensure, at least, recursive feasibility. To account for this uncertainty due to the approximation error, we are faced with solving a robust optimal control problem that has several common features with \eqref{eq:baseline robust MPC} at the level of cost and constraint specifications, but differs at the level of the dynamics. With this motivation, we define the discrete-time controlled system where the approximation noise appears as an uncertainty:
\begin{align}\label{eq:noisy-system}
    \st_{t+1} =
	 \pfield \bigl(\st_t, \ut_t, (w_t,v_t)\bigr),\quad x_0\text{ given},\quad t \in \Nz,
\end{align}
where 
\begin{enumerate}[label=\textup{(4-\alph*)}, leftmargin=*, widest=b, align=left]
	\item \label{eq:noisy-data1_1} the right-hand side of \eqref{eq:noisy-system} is the mapping 
		\begin{multline*}
		    \Mbb \times \Ubb \times \Wbb \times \Vbb \ni \bigl(\dummyx,\dummyu,(\dummyw,\zeta) \bigr) \mapsto
			\pfield\bigl(\dummyx,\dummyu,(\dummyw,\zeta) \bigr) \Let \field(\dummyx, \dummyu+\zeta, \dummyw) \in \R[\dimsp],\nn
		\end{multline*}
		which is continuous because \(\field\) is continuous in view of \eqref{eq:system-dynamics}, and the uniform approximation error margin \(\eps > 0\) has been pre-specified;\footnote{Recall that this error margin \(\eps > 0\) is the choice of the designer.}
	\item \label{eq:noisy-data1} the data \eqref{eq:system-st-con-dist}--\eqref{eq:constraint-sets-wt} carry over to \eqref{eq:noisy-system} with \(\pfield\) in place of \(\field\) in \eqref{eq:system-dynamics} modulo obvious changes;
	\item \label{eq:noisy_baseline_data} the data \eqref{eq:system:horizon}--\eqref{eq:system:policies} from the baseline MPC problem \eqref{eq:baseline robust MPC} are satisfied. 
\end{enumerate}
With the preceding ingredients, our \embf{approximation-ready robust optimal control problem} for the synthesis of receding horizon control is given by
\begin{equation}
	\label{eq:approx-ready robust MPC}
	\begin{aligned}
		& \inf_{\policy(\cdot)} \sup_{W} && \sum_{t=0}^{\horizon-1} \cost(\dummyx_t, \dummyu_t) + \fcost(\dummyx_{\horizon})\\
		& \sbjto && \begin{cases}
			\dummyx_{t+1} = \pfield\bigl(\dummyx_t, \dummyu_t, (\dummyw_t, \zeta_t)\bigr),\dummyx_0 = \xz,\\
	\dummyx_t\in\Mbb,\,\dummyx_\horizon\in\admfinst,\text{ and }\dummyu_t+\dummyv_t\in\admact\\
			\quad\text{for all }(\dummyw_t, \zeta_t,t)\in\Wbb\times\Vbb \times \timestamp{0}{\horizon-1},\\
			\dummyu_t = \policy_t(\dummyx_t),\,
			\policy(\cdot)\in\policies, \\ W \Let \bigl((\dummyw_0, \zeta_0), \ldots, (\dummyw_{\horizon-1}, \zeta_{\horizon-1})\bigr).
		\end{cases}
	\end{aligned}
\end{equation}
Overloading our notation a little, we continue to denote the set of feasible initial states of the optimal control problem \eqref{eq:approx-ready robust MPC} by \(X_N\). We enforce the following assumption:
\begin{assumption}
	\label{a:approx-ready exun}
	The set \(\fset\) corresponding to problem \eqref{eq:approx-ready robust MPC} is nonempty and the problem \eqref{eq:approx-ready robust MPC} admits a unique solution \(\opt{\policy}(\cdot) = (\opt{\mu}_t(\cdot))_{t=0}^{\horizon-1}\) for each \(\xz \in \fset\).
\end{assumption}
\begin{remark}[Motivation, robust formulation and uniform error]
\label{r:robust and uniform}
The approximation-ready robust optimal control problem \eqref{eq:approx-ready robust MPC} relates to the system \eqref{eq:system} in the following way: the fictitious small noise (\(v_t\)) in \eqref{eq:noisy-system} accounts for the uncertainty (noise) in the control actions that enter due to our approximation procedure.
In our results, the approximation margin \(\eps\) is the choice of the designer, and this margin appears in the definition of the mapping \(\pfield\) --- its domain involves the set \(\Vbb\). The problem \eqref{eq:approx-ready robust MPC} thereby accounts for all the uncertainties that could have entered in \eqref{eq:system} by bootstrapping the noisy term at the \emph{synthesis stage}. See also Remark \ref{rem:robsut_approach}. 

\end{remark}

Of course, the first element \(\upopt(\cdot)\) of the policy solving \eqref{eq:approx-ready robust MPC} is of relevance in the context of MPC. Our approximation procedure, to be described in the sequel, produces an approximate feedback \(\mutrunc(\cdot)\) in place of \(\upopt(\cdot)\) that satisfies
\begin{equation}
	\label{eq:policy bound}
	\unifnorm{\mutrunc(\cdot) - \upopt(\cdot)} \le \eps.
\end{equation}
Equipped with our approximate feedback \(\mutrunc(\cdot)\), our \embf{approximate explicit MPC strategy} for \eqref{eq:system} is encoded by the following two steps:
\begin{itemize}[label=\(\circ\)]
	\item measure the states \(\st_t\) at time \(t\),
	\item apply \(\con_t = \mutrunc(\st_t)\), increment \(t\) to \(t+1\), and repeat.
\end{itemize}
Since the uncertainty due to the approximation error \eqref{eq:policy bound} has been accounted for in \eqref{eq:approx-ready robust MPC}, employing the control actions \(\con_t = \mutrunc(\st_t)\) for each \(t\) in \eqref{eq:system} ensures that all the given constraints are satisfied and recursive feasibility is guaranteed. We refer the readers to the Assumption \ref{assump:stability_assumptions} and the Theorem \ref{thrm:stability_main_result} ahead in \S\ref{sec:main_tech_results} for closed-loop stability and recursive feasibility guarantees. The next section provides a quick overview of the chief approximation tool we employ in this manuscript.

%% file: verArXiVmath.tex
\section{Review of quasi-interpolation}\label{appendix_app_app_theory}
We provide a summary of and a few relevant results on a particular class of quasi-interpolation technique (known as \emph{approximate approximation}) which is our chief approximation tool. Let \(d \in \N\), \(m\in \Z^d\), \(h>0\) be the discretization parameter, \(\Dd>0\), and \(\aset[]{mh \suchthat m \in \Z^d}\) be a set of data points specified on a cardinal square grid of dimension \(d\). Let \(\psi(\cdot)\)  be a continuous ``generating" function. The quasi-interpolation scheme corresponding to a continuous function \(u:\Rbb^d \lra \Rbb\) is given by
\begin{align}\label{eq:gen_quasi_sdim}
\Rbb^d \ni x \mapsto \widehat{u}(x) \Let \frac{1}{\Dd^{d/2}}\sum_{m \in \Z^d}u(mh)\,\psi \left(\frac{x - mh}{h \sqrt{\Dd}}\right)
\end{align}
for all \(x \in \Rbb^d\). The generating function \(\psi:\Rbb^d \lra \Rbb\) in \eqref{eq:gen_quasi_sdim} belongs to the Schwartz class of functions that needs to satisfy the properties mentioned below: let \(\alpha \in \mathbb{N}^d\) denote a multi-index of length \([\alpha] \Let \alpha_1 + \ldots+ \alpha_d\), and we set \(z^{\alpha} \Let z_1^{\alpha_1} \cdots z_d^{\alpha_d}\) for \(z \in \Rbb^d\). The usual \(\alpha\)-order derivative of \(u(\cdot)\) is denoted by
\begin{align}
	\partial^{\alpha}u(x) \Let \frac{\partial^{[\alpha]}}{\partial x_1^{\alpha_1} \cdots \partial x_d^{\alpha_d}} u(x)\,\,\text{for }x \in \Rbb^d. \nn
\end{align}
The generating function \(\psi(\cdot)\) satisfies: 
\begin{itemize}[ leftmargin=*, widest=b, align=left]
\item \emph{the continuous moment condition of order \(M\)}, i.e.,
\begin{align}\label{eq:moment_condition}
	&\int_{\Rbb^d}\psi(y)\,\dd y=1\text{ and } \int_{\Rbb^d}y^{\alpha}\psi(y)\,\dd y=0 \nn\\& \,\,\text{for all}\,\, \alpha,\,1\le [\alpha] < M;
\end{align}
\item \emph{the decay condition}: For all \(\alpha \in \mathbb{N}^d\) satisfying \(0 \le [\alpha] \le \lfloor d/2 \rfloor + 1\), the function \(\psi(\cdot) \) is said to satisfy the decay condition of exponent \(K\) if there exist \(C_0>0\) and \(K>d\) such that 
\begin{align}\label{eq:decay_condition}
   \left(1+\|x\|\right)^K  |\partial^{\alpha} \psi(x)| \le C_0 \quad\text{for}\,\, x\in \Rbb^d.
\end{align}
\end{itemize}

\begin{table}[htbp]
\centering
\begin{tabular}{lc}
\toprule
Generating function \(\genfn(x)\) & Order (\(M\))  \\
    \midrule
 \(\genfn_1(x)=\frac{1}{\sqrt{\pi}}\epower{-x^2}\) \vspace{2mm} & \(2\)  \\ 
\(\genfn_2(x)=\frac{1}{\sqrt{\pi}}\left(\frac{15}{8}-\frac{5}{2}x^2+\frac{x^4}{4}\right)\epower{-x^2}\) \vspace{2mm} & \(6\) \\
\(\genfn_3(x)=\frac{1}{\sqrt{\pi}}\bigl(\frac{315}{128}-\frac{105}{16}x^2+\frac{63}{10}x^4-\frac{3}{4}x^6\) & \(10\) \\
\hspace{8mm}\(+\frac{1}{24}x^8\bigr)\epower{-x^2}\) \vspace{2mm}& \\
\(\genfn_4(x) = \sqrt{\frac{\mathrm{e}}{\pi}}\epower{-x^2}\cos \sqrt{2}x\) \vspace{2mm} & \(4\) \\ 
\(\genfn_5(x) = \frac{1}{\pi}\text{sech } x\) & \(2\) \\
        \bottomrule
    \end{tabular}
    \caption{The functions \(\R[]\ni x \mapsto \genfn_i(x), \, i \in \aset{2,3}\), are the higher order \emph{generalized Gaussian} or the \emph{Laguerre Gaussian}; \(\R[]\ni x \mapsto \genfn_4(x)\) is a \emph{trigonometric Gaussian}, and \(\R[]\ni x \mapsto \genfn_5(x)\) is the \emph{hyperbolic secant} \cite[Chapter 3]{ref:mazyabook}.}
  \label{tab:basis_specifications}
\end{table}
Table \ref{tab:basis_specifications} records a few candidate one-dimensional generating functions. Now we state the key estimate that we employ in this article in the context of Lipschitz continuous policies: 
 \begin{theorem}{\cite[Theorem 2.25]{ref:mazyabook}}
 \label{thrm:Holder-Lipschitz-estimate}
Consider a Lipschitz continuous function \(u:\Rbb^d \lra \Rbb\) of Lipschitz rank \(L_0\), i.e., \(u(\cdot)\) satisfies the inequality \(\|u(x+y)-u(x)\| \le L_0 \|y\|\) for all \(x,y \in \Rbb^d\). Let \(h>0\) and suppose that \(\aset[]{mh \suchthat m \in \Z^d} \subset \Rbb^d\) be the set of data sites for \(u(\cdot)\). In addition, suppose that the generating function \(\psi(\cdot)\) satisfies the moment condition \eqref{eq:moment_condition} of order \(M\), and the decay condition \eqref{eq:decay_condition} with exponent \(K\), and \(\mathcal{F} \psi(0)=1\), \(\mathcal{F}\) being the Fourier transform operator on \(\Rbb^d\). Then 
 \begin{align}\label{eq:Holder-Lip_estimate}
    \unifnorm{\widehat{u}(\cdot)- u(\cdot)} \le C_{\gamma}L_0 h \sqrt{\Dd}+\Delta_0(\psi,\Dd),
\end{align}
where \(\Delta_0(\psi,\Dd) \Let \mathcal{E}_0(\psi,\Dd)\unifnorm{u(\cdot)}\) is the saturation error, \(C_{\gamma} \Let M \cdot \Gamma (M)/\Gamma(M+2)\) is a constant, and the term
\(\mathcal{E}_0(\psi,\Dd)\) is given by 
\begin{align}\label{eq:saturation_decay_term}
    \mathcal{E}_0(\psi,\Dd)(\cdot) \Let \sup_{x \in \Rbb^d}\sum_{\nu \in \Z^d\setminus \{0\}}\mathcal{F} \psi(\sqrt{\Dd}\nu)\epower{2\pi i \inprod{x}{\nu}}.
\end{align}
\end{theorem}
\begin{remark}\label{rem:corollary_2_13_Maz}
The assumptions on \(\psi(\cdot)\) guarantee that for any preassigned \(\eps'>0\), we can choose \(\Dd_{\mathrm{min}}>0\) such that for any \(\Dd \ge \Dd_{\mathrm{min}}\), \(\mathcal{E}_0(\psi,\Dd)\le \frac{\eps'}{\bar{c}\unifnorm{u(\cdot)}}\) where \(\bar{c}>0\) is a constant which can be adjusted to make the total approximation error (after fixing a suitable \(h\)) bounded above by \(\eps'\). Readers are referred to \cite[Chapter 2, Corollary 2.13]{ref:mazyabook} for additional details.
\end{remark}
\begin{remark}\label{rem:truncated_sum}
Notice that the approximation formula \eqref{eq:gen_quasi_sdim} involves an infinite sum over a \(d\)-dimensional integer lattice to approximate the function \(u(\cdot)\) at a point \(x \in \Rbb^d\), and therefore, an infinite number of summands plays a part in constructing the approximant \(\widehat{u}(\cdot)\). However, this sum can be truncated in most applications due to the sharp decay property that the generating function \(\psi(\cdot)\) enjoys. Define
\begin{align}\label{eq:finite_grid}
    \finset_x(\Lambda) \Let \aset[\big]{mh \suchthat mh \in \Z^d} \cap \Ball_2^d[x,\Lambda],
\end{align}
and consider the finite-sum truncation \(u(\cdot)\):
\begin{align}\label{eq:truncated_quasi_sdim}
\Rbb^d \ni x \mapsto u^{\dagger}(x) \Let \Dd^{-d/2} \hspace{-3mm}   \sum_{mh \in \finset_x(\Lambda)} \hspace{-3mm}u(mh)\,\psi \left(\frac{x-mh}{h \sqrt{\Dd}}\right)
\end{align}
for all \(x\in \Rbb^d\). The approximant \eqref{eq:truncated_quasi_sdim} considers only the points \(mh\) inside the ball \(\Ball_2^{d}[x,\Lambda]\), and thus the grid \(\finset_x(\Lambda)\) in \eqref{eq:finite_grid} is finite. The difference between the approximant \(\widehat{u}(\cdot)\) and \(u^{\dagger}(\cdot)\) can then be bounded by
\begin{align}\label{eq:truncated_bound_1}
    \| u^{\dagger}(x)-\widehat{u}(x) \| \le \mathcal{B} \biggl( \frac{\sqrt{\Dd}h}{\Lambda} \biggr)^{K-d} \unifnorm{u(\cdot)}\;\;\text{for all }x \in \Rbb^d, 
\end{align}
where 
\begin{enumerate}[label=\textup{(R-\alph*)}, leftmargin=*, widest=b, align=left]
\item \label{trunc:rem:2} \(K>d\) is the decay exponent (see \eqref{eq:decay_condition}) of \(\psi(\cdot)\), and

\item \label{trunc:rem:1} \(\mathcal{B}\) is a constant that depends on \(d\), a conservative upper bound of \(\mathcal{B}\) is \(\frac{C_0 \hat{C}}{K-d}\), where \(C_0\) is the right-hand side of \eqref{eq:decay_condition} and \(\hat{C}>0\) is a constant that depends \(\psi(\cdot)\) (see \cite[\S 4.3.2, Page 50]{weisse2009global}, \cite[\S2.3.2, Page 35]{ref:mazyabook} for concrete expressions).
\end{enumerate}

An interesting case, one that we will employ below, is when \(\Lambda \Let \rzero h\) for some \(\rzero >0\). The error caused by the truncated approximant \(u^{\dagger}(\cdot)\) is comparable to the saturation error of \(\widehat{u}(\cdot)\) \cite[Chapter 2]{ref:mazyabook}, and the bound in \eqref{eq:truncated_bound_1} is independent of the parameter \(h\). Thus, \(u^{\dagger}(x)\) takes into account only the terms for which \(\|x/h - m\| \le \rzero\), which makes the number of summands in \eqref{eq:truncated_quasi_sdim} independent of the step-size \(h.\) 
\end{remark}

%% file: verArXiVresults.tex

\section{Main result: Algorithm and theory}\label{sec:main_results} 

We state our main results in this section in the form of an algorithm and a closed-loop stability result under the approximate feedback policy for the system \eqref{eq:system} derived from \eqref{eq:approx-ready robust MPC}. The algorithm foreshadows the theoretical guarantees and is designed to ensure recursive feasibility despite the approximation errors that creep into the explicit (approximate) feedback constructed herein via quasi-interpolation. More specifically, the algorithm provides a systematic way to extend (whenever needed) the domain of the approximation-ready feedback policy \(\upopt(\cdot)\) and approximate it in the \emph{uniform} sense \emph{for all feasible} initial data. 




\subsection{Lipschitz extension algorithm}\label{sec:extension_algo} The algorithm can be segregated into four parts:
\label{s:extension process}


\subsubsection{Calculation of the approximation-ready policy at points}\label{sec:extension_algo_step_one}
Recall that the optimal feedback \(\upopt(\cdot)\) is generated via the approximation-ready robust MPC problem \eqref{eq:approx-ready robust MPC}. We limit our results to the case of policies that are Lipschitz continuous in the parameter \(\xz\); this is an explicit assumption of the main result --- Theorem \eqref{thrm:stability_main_result} ahead. Regularity of this nature is important for the interpolation technique we shall employ in the sequel. Any numerical algorithm to solve constrained minmax optimization problems can be applied to solve the optimization problem \eqref{eq:approx-ready robust MPC} and generate \(\upopt(\cdot)\) at grid points of a uniform cardinal grid on the state space \(\Mbb\).

\subsubsection{Extension to \(\Rbb^d\)}
We present a technique to extend the domain of the map \(\upopt(\cdot)\) to all \(\Rbb^d\) so that the estimates in \S\ref{appendix_app_app_theory} can be employed. The procedure also ensures good numerical fidelity and very high accuracy. We employ a simple extension algorithm which is along the same lines as \cite[Appendix 3]{ref:Loja_real}. To this end, define an \(h\)-net \cite{ref:hnets} of \(X_N\) given by \(\widehat{X}_N \Let \fset \cap \aset[]{mh \suchthat m \in \Z^d}\). Let us define the extended policy by 
\begin{align}\label{eq:extented_policy}
   \Rbb^d \ni x \mapsto \mu_{E}\as(x) \Let \inf_{y \in \widehat{X}_N} \bigl( \upopt(y) + L_0 \norm{x-y}_2 \bigr).
\end{align}
Note that this extension is performed component-wise for each of the \(\dimcon\) components of the function \(\mu_{E}\as(\cdot)\). 
\begin{proposition}\label{prop:extended_lipcon}
If the policy \(\upopt(\cdot)\) is Lipschitz continuous with Lipschitz rank \(L_0\), then the extended policy \(\mu_E\as:\Rbb^d \lra \Rbb^{\dimcon}\) is also Lipschitz with the same Lipschitz rank \(L_0\). 
\end{proposition}
\begin{proof}
First note that the policies \(\upopt(\cdot)\) and \(\mu_E\as(\cdot)\) coincide on \(\widehat{X}_N\). Let \(x,x' \in \Rbb^d\), for any \(y \in \widehat{X}_N\) we have
\begin{align}
   \upopt(y)+L_0 \norm{x-y}_2 &\le \upopt(y)+L_0 \norm{x-x'+x'-y}_2 \nn \\& \le   \upopt(y)+L_0 \norm{x'-y}_2 + L_0 \norm{x-x'}_2. \nn
\end{align}
Taking infimum over \(y\in \widehat{X}_N\) on both sides gives us
\begin{align}
    \mu_E\as(x) \le \mu_E\as(x')+L_0\norm{x-x'}_2;
\end{align}
interchanging the roles of \(x\) and \(x'\) one see that \(\norm{\mu_E\as(x) - \mu_E\as(x')}_2 \le \norm{x-x'}_2.\)
The proof is complete.
\end{proof}
In the sequel, we shall overload notation and continue to label the policy \(\mu_E\as(\cdot)\) after extension to \(\Rbb^d\) as \(\upopt(\cdot)\) itself. Empirical evidence (to be given in Section \eqref{sec:num_exp}) suggests that this extension step may be skipped in certain numerical examples without transgressing the error bounds.

\subsubsection{Approximation}\label{sec:extension_algo_step_three}
The parent approximation engine is given, as in \eqref{eq:gen_quasi_sdim}, by the interpolation formula
\begin{align}\label{eq:parent_approx_policy}
\hspace{-2mm}\Rbb^d \ni x \mapsto   \appr{\mu}_0(\st) \Let \Dd^{-d/2}   \sum_{ m \in \Z^d}\hspace{-2mm}\upopt(mh)\,\psi \left(\frac{\st-mh}{h \sqrt{\Dd}}\right) 
\end{align}
for all \(x \in \Rbb^d\). We use the following truncated version of the parent quasi-interpolation scheme \eqref{eq:parent_approx_policy} (see Remark \eqref{rem:truncated_sum} for more details and our motivation behind the employment of \eqref{eq:approx_policy}) for the extended feedback policy \(\upopt(\cdot)\), given by:
\begin{align}\label{eq:approx_policy}
\Rbb^d \ni x \mapsto   \mutrunc(\st) \Let \Dd^{-d/2}   \sum_{mh \in \mathbb{F}_x(\rzero)}\upopt(mh)\,\psi \left(\frac{\st-mh}{h \sqrt{\Dd}}\right)
\end{align}
for all \(x \in \Rbb^d\), and \(\finset_x(\rzero)\) is as defined in \eqref{eq:finite_grid} (with \(\Lambda \Let \rzero h\)). 
The extended feedback policy \(\upopt(\cdot)\) is Lipschitz continuous by Proposition \ref{prop:extended_lipcon}, and is defined over the whole space \(\Rbb^d\). As a result, the following estimate holds (see Theorem \eqref{thrm:Holder-Lipschitz-estimate}): 
\begin{align}\label{eq:lipschitz_estimate_mpc}
	\unifnorm{\apprfb(\cdot) - \upopt(\cdot)} \le C_{\gamma}L_{0}h\sqrt{\Dd}+ \Delta_0(\psi,\Dd).
\end{align}
The term \(\Delta_0(\psi,\Dd)\) is the \emph{saturation error}.

Three quantities in \eqref{eq:approx_policy} --- \(h\), \(\Dd\), and \(\rzero\) --- need to be picked at this stage, \embf{depending on the prescribed error margin}. To this end, fix a desired uniform error margin \(\eps > 0\). We proceed to dominate the left-hand side of \eqref{eq:lipschitz_estimate_mpc} by \(\eps\) in \embf{three} steps:

\begin{itemize}[leftmargin=*, label=\(\triangleright\)]
	\item On the right hand side of \eqref{eq:lipschitz_estimate_mpc}, the second term \(\Delta_0( \psi, \Dd)\) --- the \emph{saturation error} --- depends on the shape parameter \(\Dd\), and can be reduced below \(\frac{\eps}{3}\) by increasing \(\Dd\). Notice that this term is \emph{independent} of \(h\), and therefore this step can be carried out by means of increasing \(\Dd\) alone.
	\item The first term on the right hand side of \eqref{eq:lipschitz_estimate_mpc} converges to zero, for every fixed \(\Dd\), as \(h\to 0\). Thus after fixing \(\Delta_0(\psi,\Dd)\) in the preceding step (which ensures \(\Delta_0(\psi,\Dd) \le \eps/3\)), we pick \(h\) such that the first term is dominated by \(\frac{\eps}{3}\).
	\item We pick \(\rzero>0\) such that the error between the truncated and the parent approximants (see \eqref{eq:truncated_bound_1}) is below \(\frac{\eps}{3}\).
\end{itemize}
The total \embf{uniform} error, consequently, stays within the preassigned bound \(\eps\); Theorem \ref{thrm:stability_main_result} ahead describes how to choose the parameters \(\rzero\), \(h\), and \(\Dd\).


\subsubsection{Restriction}\label{sec:extension_algo_step_four}
Finally, we restrict the approximated policy \(\mutrunc(\cdot)\) to the set \(\fset\). By construction \(\mutrunc(\cdot)\) satisfies \(\|\upopt(x)-\mutrunc(x)\|\le \eps\) for all \(x \in \fset\).


       
       
        
        
        

\subsection{Stability guarantees under the Lipschitz extension algorithm} \label{sec:main_tech_results}
Let us recall that the approximation-ready feedback policy for \eqref{eq:approx-ready robust MPC} is \(\upopt(\cdot)\) and the approximate feedback policy is \(\mutrunc(\cdot)\) (obtained from \(\upopt(\cdot)\) by following the steps in \secref{s:extension process}). Let us establish conditions for stability of \eqref{eq:system} under the approximate policy \(\mutrunc(\cdot)\). 
\begin{assumption}\label{assump:stability_assumptions}
For robust stability we need \(\pfield(\cdot), \fcost(\cdot), \admfinst\), and \(\cost(\cdot,\cdot)\) to satisfy following properties \cite{ref:MayFal-19}:  
\begin{itemize}[label=\(\circ\), leftmargin=*]

\item \label{eq:stability_prop1} For all \(\xi \in \admfinst\), there exists a feedback \(\dummyx \mapsto \dummyu(\dummyx) \Let \mu_F (\dummyx)  \in \Ubb\) such that 
	\[
		\pfield\bigl( \dummyx,\mu_F(\dummyx),(\dummyw,\dummyv) \bigr) \in \admfinst\quad\text{for every }\dummyx \in \admfinst, (\dummyw,\dummyv) \in \Wbb \times \Vbb.
	\]

\item\label{eq:stability_prop2} There exists a number \(b>0\) such that: \\
\(\fcost \circ \pfield \bigl( \dummyx,\mu_{F}(\dummyx),(\dummyw,\dummyv) \bigr) -\fcost(\dummyx) \le - \cost\bigl(\dummyx,\mu_{F}(\dummyx)\bigr)+b \,\,\text{for every}\,\,\dummyx \in \admfinst, (\dummyw,\dummyv) \in \Wbb \times \Vbb.\)

\item The terminal set \(\admfinst \subset \Mbb\) is compact and contains the origin in its interior, and there exist \(\mathcal{K}_{\infty}\) (See \cite[Definition 2.13]{ref:GruPan-17}) functions \(\alpha_1(\cdot), \alpha_2(\cdot)\) such that

\begin{itemize}[label=\(\triangleright\), leftmargin=*]

\item \(\cost(\xi,\dummyu) \ge \alpha_1 \bigl(|\xi|\bigr)\) for every \(\xi \in \fset\) and for every \(\dummyu \in \Ubb\),  \((\dummyw,\dummyv) \in \Wbb \times \Vbb\),

\item \(\fcost(\xi) \le \alpha_2 \bigl(|\xi|\bigr) \,\,\) for every \(\xi \in \admfinst\).

\end{itemize}

\end{itemize}

\end{assumption}
Under Assumption \eqref{assump:stability_assumptions}, it can be shown \cite[\S3, Assumption 3, and the discussion thereafter]{ref:MayFal-19} that  under the receding horizon policy \(\upopt(\cdot)\), the closed-loop system of \eqref{eq:approx-ready robust MPC} generated by the dynamical system \(\st_{t+1}= \pfield\bigl(\st_t, \upopt(\st_t), (\dist_t,v_t)\bigr)\) is robustly stable and the value function \(\valuefunc(\cdot)\) satisfies the following descent property for some \(b\in\R\):
\begin{align}
	\label{eq:value_descent}
	\valuefunc (\st_1) - \valuefunc (\st_0) \le - \cost \bigl(\st,\upopt(\st_0)\bigr)+b
\end{align} 
for every \(\st_0, \st_1 \in \fset\), the set of all feasible states for which \eqref{eq:approx-ready robust MPC} admits a solution. Figure \eqref{fig:stab_flow} explains the interplay between the optimization problems \eqref{eq:baseline robust MPC} and \eqref{eq:approx-ready robust MPC} and how Assumption \eqref{assump:stability_assumptions} comes into the picture.
\begin{figure}[ht]
\centering
\includegraphics[scale=0.6]{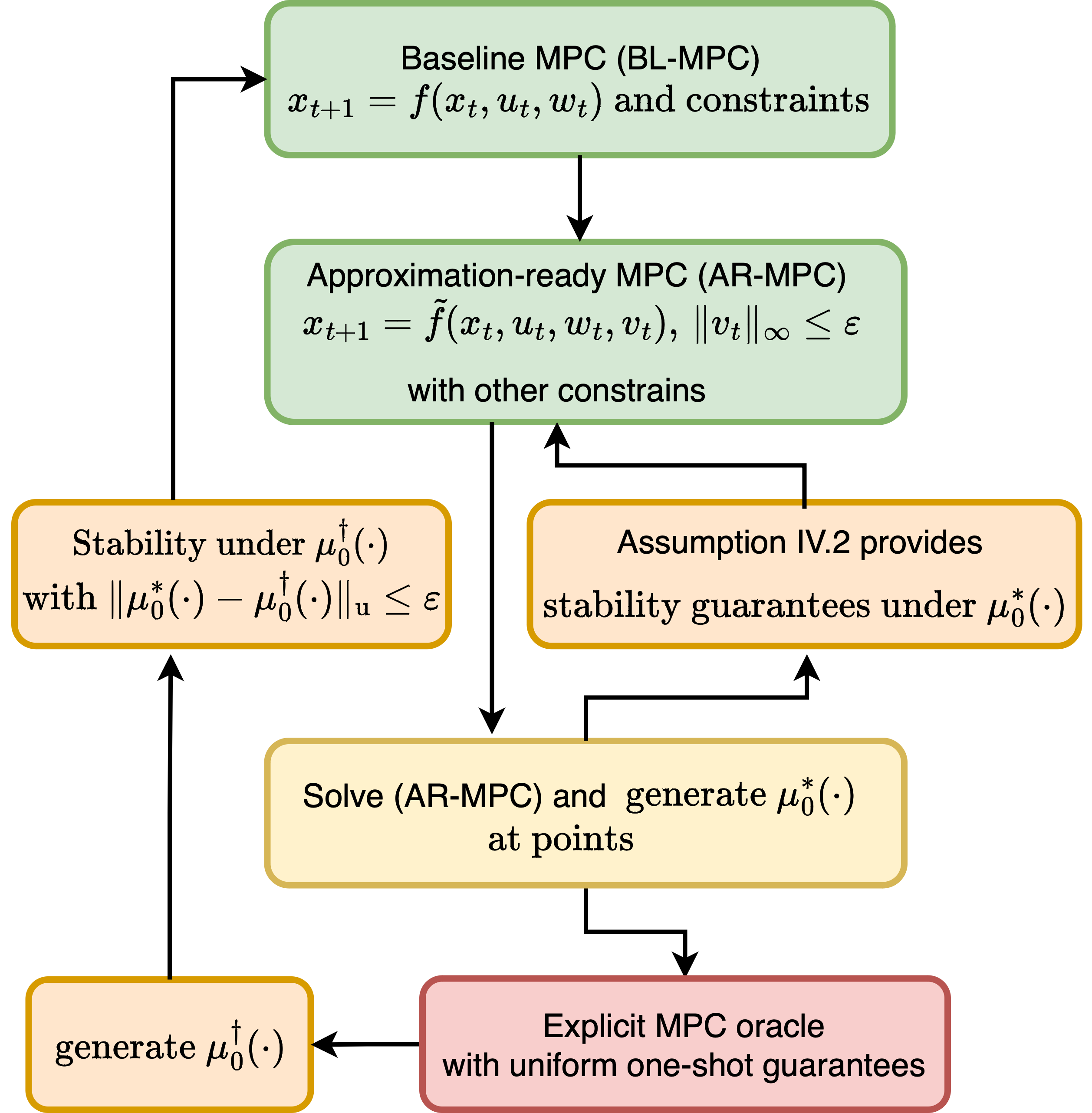}
\caption{A flowchart explaining the QuIFS algorithm.}\label{fig:stab_flow}
\end{figure} 
Against this backdrop, here is our key technical result concerning the approximate feedback policy \(\mutrunc(\cdot)\).
\begin{theorem}\label{thrm:stability_main_result}
Consider the constrained optimal control problem \eqref{eq:approx-ready robust MPC} along with its associated data \eqref{eq:noisy-data1_1}--\eqref{eq:noisy_baseline_data} and suppose that Assumption \eqref{a:approx-ready exun} holds. Let \(\upopt(\cdot)\) be the unique approximation-ready receding horizon policy. Assume that \(\fset \ni x \mapsto \upopt(x) \in \Ubb\) is Lipschitz continuous with Lipschitz constant \(L_{0}\). Then:
\begin{enumerate}[label=\textup{(\ref{thrm:stability_main_result}-\alph*)}, leftmargin=*, widest=b, align=left]
	\item \label{thrm:main:estimate} For every given \(\eps>0\), there exist a generating function \(\psi(\cdot) \in \mathcal{S}(\Rbb^d)\), a pair \((h,\Dd) \in \loro{0}{+\infty}^2\), and \(\rzero >0\), such that the approximate feedback map \(\fset \ni \st \mapsto \mutrunc(\st) \in \Ubb\) defined in \eqref{eq:approx_policy} is within a uniform error margin \(\varepsilon\) from \(\upopt(\cdot)\); to wit,
		\begin{equation}
		    	\|\upopt(x) - \mutrunc(x)\| \le \eps \quad \text{for all}\,x \in \fset.  \nn
		\end{equation}
	\item \label{thrm:main:stability} In addition, suppose that Assumption \eqref{assump:stability_assumptions} holds. Then under the approximate feedback map \(\fset \ni \st \mapsto \mutrunc(\st)\) constructed via Algorithm \eqref{alg:extension_algo}, the system \eqref{eq:system} is ISS-like stable in the sense of Definition \eqref{appen_b_Regional_Practical_ISS}. 
\end{enumerate}
\end{theorem}
\begin{proof}
We begin by giving a proof of the first assertion. Fix \(\eps>0\), \(\psi(\cdot) \in \mathcal{S}(\Rbb^d)\) satisfying moment condition of order \(M\) and decay condition of order \(K>d\) with \(C_0\) as the upper bound in \eqref{eq:decay_condition}. The approximate feedback policy \(\mutrunc(\cdot)\) derived from the extended approximation-ready policy \(\upopt(\cdot)\) is given by the summation 
\begin{align}\label{eq:truncated_quasi_d_dim_proof}
x \mapsto     \mutrunc(x) \Let \Dd^{-d/2}   \sum_{\mathclap{\substack{mh \in \finset_x(\rzero)}}}\upopt(mh)\,\psi \left(\frac{x-mh}{h \sqrt{\Dd}}\right)
\end{align}
for \(x \in \Rbb^d\), and \(\finset_x(\rzero)\) as defined in \eqref{eq:finite_grid}. By assumption, the approximation-ready policy \(\upopt(\cdot)\) is Lipschitz continuous with Lipschitz rank \(L_0\) and so is the extended policy \(\upopt(\cdot)\) (we overload the notation here) with the same Lipschitz rank \(L_0\) (see Proposition \ref{prop:extended_lipcon} in \S\ref{sec:extension_algo}). Let \(\appr{\mu}_0(\cdot)\) be as given in \eqref{eq:parent_approx_policy}. Then, we have the estimate (see Theorem \eqref{thrm:Holder-Lipschitz-estimate})
\begin{align}\label{eq:lipschitz_estimate_mpc_proof}
	\unifnorm{\apprfb(\cdot) - \upopt(\cdot)} \le C_{\gamma}L_{0}h\sqrt{\Dd}+ \Delta_0(\psi,\Dd).
\end{align}
From \cite[Chapter 2, Corollary 2.13]{ref:mazyabook} it follows that for the preassigned \(\eps>0\), we can find \(\Dd_{\mathrm{min}}>0\) such that whenever \(\Dd \ge \Dd_{\mathrm{min}}\), we have
\begin{align}
	\label{epsilon_0_estimate}
	\mathcal{E}_0(\psi,\Dd) \le \frac{\eps}{3\unifnorm{\upopt(\cdot)}}.
\end{align}
We pick \(\Dd \ge \Dd_{\mathrm{min}}\), which ensures
\begin{align}
	\label{delta_0_estimate}
	\Delta_0(\psi,\Dd) \le \frac{\eps}{3}.
\end{align}
Now we fix
\begin{align}
	\label{h_estimate}
	h = \frac{\eps}{3C_{\gamma}L_{0}\sqrt{\Dd}},
\end{align}
	which leads to the first term on the right-hand side of \eqref{eq:lipschitz_estimate_mpc_proof} to be dominated by \(\frac{\eps}{3}\). Combining the estimate \eqref{delta_0_estimate} with \eqref{h_estimate}, from \eqref{eq:lipschitz_estimate_mpc_proof} we arrive at
\begin{equation}
	\label{proof:final_eps_estimate}
    \unifnorm{\upopt(\cdot)-\appr{\mu}_0(\cdot)} \le \frac{2\eps}{3}.
\end{equation}
Notice that the estimate \eqref{proof:final_eps_estimate} is valid after \(\upopt(\cdot)\) has been extended to \(\Rbb^d\). Let \(\mathcal{B}\) be a constant specific to \(\psi(\cdot)\) as given in \eqref{trunc:rem:1} of Remark \eqref{rem:truncated_sum}; in \eqref{eq:truncated_bound_1} with \(\Lambda \Let \rzero h\), we pick
\[ \rzero \Let \sqrt{\Dd}\biggl(\frac{\eps}{3\mathcal{B} \unifnorm{\upopt(\cdot)}}\biggr)^{1/(d-K)}.\]
Then \(\|\appr{\mu}_0(x) - \mutrunc(x)\| \le \frac{\eps}{3}\) for all \(x \in \Rbb^d\). Now restricting the domains of \(\upopt(\cdot)\), \(\appr{\mu}_0(\cdot)\), and \(\mutrunc(\cdot)\) to \(\fset\) while retaining the same notation for all of them, we see that
\begin{align}\label{pf:final_error_estimate}
   & \|\upopt(x)-\mutrunc(x)\|  \le \|\upopt(x)-\appr{\mu}_0(x)\| + \| \appr{\mu}_0(x) - \mutrunc(x)\| \nn\\& \le C_{\gamma}L_0h\sqrt{\Dd}+\Delta_0(\psi,\Dd)+\mathcal{B} \biggl(\frac{\sqrt{D}}{\rzero}\biggr)^{K-d}\unifnorm{\upopt(\cdot)} \nn \\& \le \frac{2\eps}{3}+\frac{\eps}{3}= \eps.
\end{align}
In summary, since \(\eps>0\) was preassigned and we picked \(\psi(\cdot) \in \mathcal{S}(\Rbb^d)\) and the triplet \((h,\Dd,\rzero) \in \loro{0}{+\infty}^3\) such that the estimate \eqref{pf:final_error_estimate} holds, the first assertion \eqref{thrm:main:estimate} stands established. 

We proceed to prove the second assertion \eqref{thrm:main:stability} concerning ISS-like stability of the closed-loop system corresponding to the system \eqref{eq:system} under the approximate feedback \(\mutrunc(\cdot)\). Under \(\mutrunc(\cdot)\), the closed-loop process is given by:\begin{align}\label{proof:closed_loop_1}
        x_{t+1} = \field \bigl(x_t,\mutrunc(x_t),w_t\bigr),
\end{align}
Recall that the (state-dependent) approximation noise is given by \(v_t \Let \mutrunc(x_t)- \upopt(x_t)\). Then from \eqref{proof:closed_loop_1} we have
\begin{align}\label{proof:closed_loop_2}
x_{t+1} &= \field \bigl(x_t,\mutrunc(x_t),w_t\bigr) \nn \\& = \field \bigl(x_t,\mutrunc(x_t)-\upopt(x_t)+\upopt(x_t),w_t\bigr) \nn \\
& = \field \bigl(x_t,v_t+\upopt(x_t),w_t\bigr) \nn \\&  = \pfield \bigl(x_t,\upopt(x_t),(w_t,v_t)\bigr),
\end{align}
where \(\pfield(\cdot)\) has been defined in \eqref{eq:noisy-system}. With the stability Assumption \eqref{assump:stability_assumptions} in place by hypothesis, for the problem \eqref{eq:approx-ready robust MPC} we have the following descent property concerning the value function (quoted in \eqref{eq:value_descent}): there exists \(b > 0\) such that the inequality
\begin{align}
    	\label{proof:value_ineq}
        \valuefunc \bigl(x_{t+1}\bigr) - \valuefunc \bigl(x_t \bigr) \le - \cost \bigl(x_t,\upopt(x_t)\bigr) + b.
\end{align}
holds. Consequently, the closed-loop system under \(\upopt(\cdot)\) of \eqref{eq:noisy-system}, i.e., the dynamics \eqref{proof:closed_loop_2}, is ISS-like stable in the sense of Definition \eqref{appen_b_Regional_Practical_ISS} and the ensuing optimal control problem \eqref{eq:approx-ready robust MPC} is recursively feasible; see \cite[\S3]{ref:MayFal-19}. This immediately proves ISS-like stability of the original controlled system \eqref{eq:system} under the approximate feedback policy \(\mutrunc(\cdot)\) (i.e., ISS-like stability of the system \eqref{proof:closed_loop_1}) in the sense of Definition \eqref{appen_b_Regional_Practical_ISS}, completing the proof.
\end{proof}
The entire procedure of extension and approximation described in \S\ref{sec:extension_algo} and \S\ref{sec:main_tech_results} is recorded in  Algorithm \eqref{alg:extension_algo}. 
\begin{algorithm2e}[!h]
\DontPrintSemicolon
\SetKwInOut{ini}{Initialize}
\SetKwInOut{giv}{Data}
\SetKwInOut{ext}{Extend}
\SetKwInOut{interpol}{Interpolation}
\giv{\(\upopt(\cdot)\) on \(\widehat{X}_N\)}
\ini{Lipschitz constant \(L_{0}\) of the policy \(\upopt(\cdot)\)}

\ext{Extend \(\upopt(\cdot)\) to whole \(\Rbb^d\) using \eqref{eq:extented_policy}}

\interpol{\(\circ\) Fix an error-margin \(\eps>0\); \\\(\circ\) Choose the tuple \\ \(\bigl(\psi(\cdot),h,\Dd,\rzero\bigr) \in \mathcal{S}(\Rbb^d) \times \loro{0}{+\infty}^3\); \\
		\(\circ\) Compute \(\mutrunc(\cdot)\) via \eqref{eq:approx_policy};\\\(\circ\) Restrict \(\mutrunc(\cdot)\) to \(\fset\).} 
\caption{Lipschitz extension and approximation}
\label{alg:extension_algo}
\end{algorithm2e}

\begin{remark}\label{rem:lip_con_estimation}
Notice that by assumption, the policy \(\upopt(\cdot)\) is Lipschitz with a known Lipschitz rank \(L_0\) and the theoretical guarantees of Theorem \ref{thrm:stability_main_result} employ \(L_0\). For implementation purposes \(L_0\) may have to be estimated via numerical techniques because, in general, an analytical expression of \(\upopt(\cdot)\), and consequently the value of \(L_0\), may not available. In our numerical study, we computed the supremum norm of the numerical gradient of \(\upopt(\cdot)\) at the uniformly spaced grid points and conservatively set \(L_0\) to be \(\widehat{L}_0 \Let 2 L_0\). Other techniques can also be employed to estimate \(L_0\) from data, e.g., as given in \cite{ref:hnets} and \cite{ref:wood1996estimation}.
\end{remark}

\begin{remark}\label{rem:x_n_h_nets}
	It is also important to note that we do not assume that the feasible set \(X_N\) is known. We do not need to calculate the feasible set for our algorithm: given a preassigned \(\eps>0\) our algorithm gives a step size \(h\), employing which we grid the state space and solve the approximation-ready robust optimal control problem \eqref{eq:approx-ready robust MPC} at each grid point. Technically speaking, we obtain an \(h\)-net \cite{ref:hnets} of \(X_N\) under the \(\infty\)-norm (`box' norm) by solving the MPC problem \eqref{eq:approx-ready robust MPC} on the uniform cardinal grid of side \(h\). Since \(h>0\) is small, we automatically get an \(h\)-approximate (in the \(\infty\)-norm) subset of \(X_N\) in this way.
\end{remark}

\subsection{Discussion}
\begin{remark}[On the robust approach]\label{rem:robsut_approach}
We reiterate that the technique of approximation (as opposed to exact evaluation) of control policies necessarily introduces uncertainties in the action variable during the operation of the underlying system. Accommodating such uncertainties at the synthesis stage naturally leads to the robust formulation of MPC irrespective of whether the original problem was nominal or robust MPC. In fact, even if stochastic modeling of uncertainties in the plant and/or measurements is considered and the resulting policies approximated by some means, still the synthesized policies must be robust with respect to the errors introduced by the approximated control policy in view of ensuring recursive feasibility. In other words, ensuring robustness in closed-loop with respect to uncertainties in the control actions is inevitable in the technique of approximation; the \emph{extent} of robustness can be \emph{prespecified} in our approach as explained above (and as pointed out in point \ref{prop:applicability_type} of \secref{sec:intro}). Among all possible types of approximation, we submit that the best choice is that of \emph{uniform} approximation; indeed, no other \(\lpL[p]\) for \(p\in \lcro{1}{+\infty}\) approximation error would guarantee boundedness of the uncertainties introduced in the control actions due to such errors, thereby rendering the robust formulation ineffective and compromising recursive feasibility. For the same reason, probabilistic guarantees of uniform approximation are insufficient by themselves, in ensuring recursive feasibility.
\end{remark}
\begin{remark}[Computational challenges]
	Of course, the computation of optimal policies in, e.g., \eqref{eq:approx-ready robust MPC} is a challenging problem. Over and above the exponential complexity introduced due to the uniform grid (pointed out in point \ref{prop:complexity} of \secref{sec:intro}), each point evaluation involves the numerical solution of a minmax problem. While the general case of nonlinear MPC offers little hope with regards to the indicated minmax computation at the present time, the linear analog (i.e., linear MPC, to be treated in \secref{sec:linear_case}) does indeed admit numerically tractable approaches in some of the most important cases. One of the early developments in this direction was reported in \cite{ref:BerBro-07}, hence the authors treated the case of the minmax problem with open-loop controls under control energy constraints and reduced it to a convex optimization program. More recently, riding on novel developments (reported in \cite{ref:DasAraCheCha-22}) on tractable techniques to solve convex semi-infinite programs, solutions to \eqref{eq:approx-ready linear robust MPC} ahead (analogs of \eqref{eq:approx-ready robust MPC}) with polyhedral constraints under affine-feedback-in-the-noise control policies (pioneered in \cite{ref:Lof-03}) have been reported in \cite{ref:GanGupCha-23}.
\end{remark}

%% file: verArXiVlinear.tex
\section{Linear MPC}
\label{sec:linear_case}

This section is devoted to linear MPC problems with the right-hand side \( \Rbb^d \times \Rbb^{\dimcon} \times \Rbb^d \ni (x_t,u_t,w_t ) \mapsto \field(x_t,u_t,w_t )\Let Ax_t + B u_t + w_t\) for each \(t \in \mathbb{N}\). Consider the linear and time-invariant dynamical system 
\begin{align}\label{eq:lin_system}
    \st_{t+1}= A\st_t+B\ut_t+ \dist_t,\quad \st_0\,\,\text{given},\,\,t \in \mathbb{N},
\end{align}
with system matrix \(A \in \Rbb^{d \times d}\) and actuation matrix \(B \in \Rbb^{d \times {\dimcon}}\). We assume that the state, control, and uncertainty constraint sets are \emph{polytopic}, each containing the respective origin in its interior. Let the cost-per-stage and the terminal cost functions are quadratic, i.e., \((\dummyx,\dummyu) \mapsto \cost(\dummyx,\dummyu) \Let \inprod{\dummyx}{Q\dummyx} + \inprod{\dummyu }{R\dummyu}\in \lcro{0}{+\infty},\) and \(\xi \mapsto \fcost(\xi)\Let \inprod{\dummyx}{P\dummyx}\in \lcro{0}{+\infty},\) with given positive (semi) definite matrices \(Q = Q^{\top} \in \Rbb^{d \times d}, R = R^{\top} \in \Rbb^{{\dimcon} \times {\dimcon}},\) and \(P=P^{\top} \in \Rbb^{d\times d}\). In addition, let the data \eqref{eq:system:terminal_set}--\eqref{eq:system:policies} continue to hold. Given these ingredients, the baseline receding horizon optimal control problem is given by:
\begin{equation}
	\label{eq:baseline linear robust MPC}
	\begin{aligned}
		& \inf_{\policy(\cdot)} \sup_{W} && \hspace{-3mm}\inprod{\dummyx_{\horizon}}{P\dummyx_{\horizon}}+\sum_{t=0}^{\horizon-1} \inprod{\dummyx_t}{Q\dummyx_t} + \inprod{\dummyu_t }{R\dummyu_t}\\
		& \sbjto && \hspace{-3mm}\begin{cases}
			\text{the dynamics}\,\,\eqref{eq:lin_system},\,\dummyx_0=\xz,\\ \dummyx_t\in\Mbb,\,\dummyx_\horizon\in\admfinst,\text{ and }\dummyu_t\in\admact\\
			\quad\text{for all }(\dummyw_t,t) \in \Wbb \times \timestamp{0}{\horizon-1},\\\dummyu_t = \policy_t(\dummyx_t),\,
			\policy(\cdot)\in\policies,\\ W \Let (\dummyw_0, \ldots, \dummyw_{\horizon-1}).
		\end{cases}
	\end{aligned}
\end{equation}
Define \(\widetilde{w}_t \Let w_t+Bv_t\) and let \(\widetilde{W} \Let \Wbb \oplus B \Vbb\). We synthesize the receding horizon control via the following linear robust optimal control problem: 
\begin{equation}
	\label{eq:approx-ready linear robust MPC}
	\begin{aligned}
		& \inf_{\policy(\cdot)} \sup_{W} && \hspace{-3mm}\inprod{\dummyx_{\horizon}}{P\dummyx_{\horizon}}+\sum_{t=0}^{\horizon-1} \inprod{\dummyx_t}{Q\dummyx_t} + \inprod{\dummyu_t }{R\dummyu_t}\\
		& \sbjto && \hspace{-3mm}\begin{cases}
			\dummyx_{t+1} = A \dummyx_t + B \dummyu_t + B \dummyv_t + \dummyw_t,\,\dummyx_0=\xz,\\
			\dummyx_t\in\Mbb,\, \dummyx_\horizon\in\admfinst,\text{ and }\dummyu_t+\dummyv_t\in\admact\\
			\quad\text{for all }(\varsigma_t,t)\in\widetilde{W}
			\times \timestamp{0}{\horizon-1},\\
			\dummyu_t = \policy_t(\dummyx_t),\,\policy(\cdot)\in\policies, W \Let (\varsigma_1,\ldots,\varsigma_{\horizon-1}).
		\end{cases}
	\end{aligned}
\end{equation}
As before, we denote the first element of the policy of the problem \eqref{eq:approx-ready linear robust MPC} by \(\upopt(\cdot)\) which is the approximation-ready receding horizon optimal policy. The following corollary mimics the Theorem \eqref{thrm:stability_main_result} in \S\ref{sec:main_results}.

\begin{corollary}\label{cor:lin_mpc_stability}
Consider the constrained optimal control problem \eqref{eq:approx-ready linear robust MPC} along with its associated data and suppose that the Assumption \eqref{a:approx-ready exun} holds. Let \(\upopt(\cdot)\) be the unique approximation-ready receding horizon policy corresponding to \eqref{eq:approx-ready linear robust MPC}. Then the following assertions hold:
\begin{enumerate}[label=\textup{(\ref{cor:lin_mpc_stability}-\alph*)}, leftmargin=*, widest=b, align=left]
\item For every given \(\eps>0\), there exist a generating function \(\psi(\cdot) \in \mathcal{S}(\Rbb^d)\), a pair \((h,\Dd) \in \loro{0}{+\infty}^2\), and \(\rzero >0\), such that the approximate feedback map \(\fset \ni \st \mapsto \mutrunc(\st) \in \Ubb\) defined in \eqref{eq:approx_policy} is within a uniform error margin \(\varepsilon\) from \(\upopt(\cdot)\); to wit,
		\begin{equation}
		    	\|\upopt(x) - \mutrunc(x)\| \le \eps \quad \text{for all}\,x \in \fset.  \nn
		\end{equation}
\item In addition, suppose that Assumption \eqref{assump:stability_assumptions} holds with \(\pfield \bigl(\dummyx,\dummyu,(\dummyw,\dummyv) \bigr) \Let A\dummyx+B \dummyu+B \dummyv+\dummyw\). Then under the approximate feedback map \(\fset \ni \st \mapsto \mutrunc(\st)\), constructed via Algorithm \eqref{alg:extension_algo}, the system \eqref{eq:lin_system} is ISS-like stable in the sense of Definition \eqref{appen_b_Regional_Practical_ISS}. 
\end{enumerate}
\end{corollary}

\begin{proof}
The feedback law \(\st \mapsto \upopt (\st)\) concerning the problem \eqref{eq:approx-ready linear robust MPC} is a continuous piecewise affine function of states \cite{ref:camacho_affinity} and is thus Lipschitz continuous. Consequently, Theorem \eqref{thrm:stability_main_result} applies and yields the proof at once.
\end{proof}

%% file: verArXiVnum.tex
\section{Numerical experiments}\label{sec:num_exp}\
In this section, we present four different numerical examples of both linear and nonlinear MPC to illustrate the applicability of the QuIFS algorithm.


\subsection{Linear MPC}


\begin{example}\label{exmp:lmpc_1}
We start with a two-dimensional example in the linear regime for two specific reasons: (a) to depict the approximation error characteristics, and (b) to show the efficacy of the Lipschitz extension procedure. To this end, for simplicity, we start with a system without any external disturbance, i.e., \(w_t = 0\) (in Example \eqref{exmp:lmpc_0} we consider a noisy dynamical system), but we synthesize (according to the QuIFS algorithm) the approximation-ready policy by translating the nominal MPC problem to a minmax problem. Consider the discrete-time linear time-invariant dynamics \cite{ref:chen2018approximating}:
\begin{equation}\label{eq:exmp_1_dyn}
\st_{t+1}= \begin{pmatrix}
1 & 0.1 \\ 0 & 1
\end{pmatrix} \st_t + \begin{pmatrix}
0.005 \\ 0.1
\end{pmatrix} \ut_t.
\end{equation}
Fix a time horizon \(\horizon \Let 15\) and consider the following finite-horizon discrete-time optimal control problem
\begin{align}\label{eq:exmp_1_linear-original problem}
   \hspace{-3mm} \begin{aligned}
        &\inf_{(\dummyu_t)_{t=0}^{\horizon-1}}  && \hspace{-3mm}\sum_{t=0}^{\horizon-1}\inprod{\dummyx_t}{Q\dummyx_t}+\inprod{\dummyu_t}{R\dummyu_t}\\
            &\sbjto && \hspace{-3mm}\begin{cases}
          \text{the dynamics }\eqref{eq:exmp_1_dyn},\,\dummyx_0=\xz,\\
          \dummyx_t \in \Mbb,\text{ and }\dummyu_t \in \Ubb \,\,\text{for all }t \in \timestamp{0}{\horizon-1},
            \end{cases}
        \end{aligned}
    \end{align}
where \(\Mbb \Let \aset[]{(\dummyx_1,\dummyx_2) \in \Rbb^2 \suchthat |\dummyx_1|\le 6,\,|\dummyx_2|\le 1}\), \(\Ubb \Let \aset[]{\dummyu \in \Rbb \suchthat |\dummyu|\le 2}\). The state weighting matrix \(Q\) is the \(2\times 2\) identity matrix and the control weighting matrix is \(R \Let 1\). The policy \(\upopt(\cdot)\) at points was obtained by gridding the state-space \([-6,6] \times [-1,1] \) with a cardinal grid of step size \(h=0.0035\). We performed our numerical computations on MATLAB 2019b using the parallel computation toolbox in an \(36\) core server with Intel(R) Xeon(R) CPU E\(5-2699\) v\(3\), \(4.30\) GHz with \(128\) Gigabyte of RAM, and we employed the solver MOSEK \cite{ref:mosek} along with the robust optimization module \cite{ref:lofberg2012automatic} in YALMIP \cite{ref:YALMIP_lofberg2004} to solve the problem \eqref{eq:exmp_1_linear-original problem} where per point computation-time was \(\sim\) \(1.5\) sec. 
It turns out that the Lipschitz constant of \(\upopt(\cdot)\) is bounded above by \(L_{0} \approx 2 \). For the quasi-interpolation scheme, we picked the Laguerre polynomial-based basis function given by:
\begin{align}\label{num:high_order_basis}
    \psi_{2M_0}(x) \Let \pi^{-d/2}\mathsf{L}_{M_0-1}^{d/2}\big(\|x\|^2\bigr)\epower{-\|x\|^2},
\end{align}
where the Laguerre polynomials are given by
\begin{align}\label{num:lagueere_pol}
    \mathsf{L}^{j}_{k}(x) \Let \frac{x^{-j}}{k!}\epower{x} \left(\frac{\dd}{\dd x}\right)^k(x^{k+j}\epower{x}), \,j>-1.
\end{align}
In \cite{ref:maz_mfld} it is shown that a \(d\)-dimensional approximant of the form \eqref{eq:approx_policy} with the basis \eqref{num:high_order_basis} leads to uniform approximation of order \(M \Let 2M_0\), i.e., \(\mathcal{O}(h^{2M_0})\), where \(M_0=1,2,\ldots\), and \(M\) is the order of the moment condition; (see \eqref{eq:moment_condition}). The Laguerre-Gaussian basis function, i.e., \eqref{num:high_order_basis} with \(d=2\) and \(M_0 = 3\) is given by
\begin{align}\label{num:sixth_order_basis_ex_1}
   \hspace{-1mm} \Rbb^2 \ni   x \mapsto \psi(x) \Let \frac{1}{\pi} \left( 3- 3 \|x\|^2+ \frac{1}{2}\|x\|^4\right)\epower{-\|x\|^2}
    \end{align}
that satisfies a moment condition of order \(M=6\), and thus \(C_{\gamma}=1/6\). For illustration, let us fix an error tolerance \(\eps= 0.005\). Fix the shape parameter \(\Dd = 2\) and simple algebra leads to the parameter \(h = \frac{\eps}{3C_{\gamma}L_{0}\sqrt{\Dd}} = 0.004\). Define \(\overline{z} \Let x-mh\). We choose \(\rzero=3\), i.e., 7 terms are used at each \(x\) in the following quasi-interpolant:
\begin{align}\label{quasi_lmpc_ex_1}
     \mutrunc(x) \Let \frac{1}{\pi\Dd}\sum_{mh \in \finset_x(\rzero)}\hspace{-3mm}\upopt(mh) \,\biggl( 3 - 3 \frac{\|\overline{z}\|^2}{h^2\Dd}+\frac{1}{2}\frac{\|\overline{z}\|^4}{h^4 \Dd^2}\biggr) \epower{-\frac{\|\overline{z}\|^2}{\Dd h^2}}.
\end{align}
It is guaranteed that for the given \(\eps= 0.005\), our one-shot synthesis produces the pair \(\bigl(h,\Dd\bigr)=\bigl(0.004,2\bigr)\) such that \(\|\upopt(\cdot)- \mutrunc(\cdot)\|_{\mathrm{u}}\le 0.005\). Figure \eqref{fig:error_lmpc_ex_1} numerically verifies this fact. Table \eqref{tab:error_margin_LMPC_ex_1} shows a list of user-defined tolerance values and the corresponding \(h\) and \(\Dd\) needed to achieve it. Although the error tolerance \(\eps\) between \(\upopt(\cdot)\) and \(\mutrunc(\cdot)\) was respected (see left-hand subfigure in Figure \ref{fig:error_lmpc_ex_1}) even before the Lipschitz extension, we employed the Lipschitz extension algorithm described in \S\ref{sec:extension_algo} to illustrate its positive effects. To this end, we set \(\rzero = 50\) and pre-calculated all the values of \(\upopt(\cdot)\) both inside \(\widehat{X}_N\) (an \(h\)-net of \(X_N\)) and outside of \(\widehat{X}_N\) via \eqref{eq:extented_policy} and subsequently we employed the approximation scheme \eqref{quasi_lmpc_ex_1} to generate \(\mutrunc(\cdot)\) and restricted it to \(\widehat{X}_N\) eventually. Compared to the left-hand subfigure of Figure \eqref{fig:error_lmpc_ex_1}, is it clear that the error reduced significantly after the extension algorithm was employed. The corresponding storage and computation-time requirements are recorded in Table \eqref{tab:metadata_ex_1}.

\begin{figure*}[h]
\centering
	\begin{subfigure}[b]{0.45\textwidth}
    \includegraphics[scale=.30]{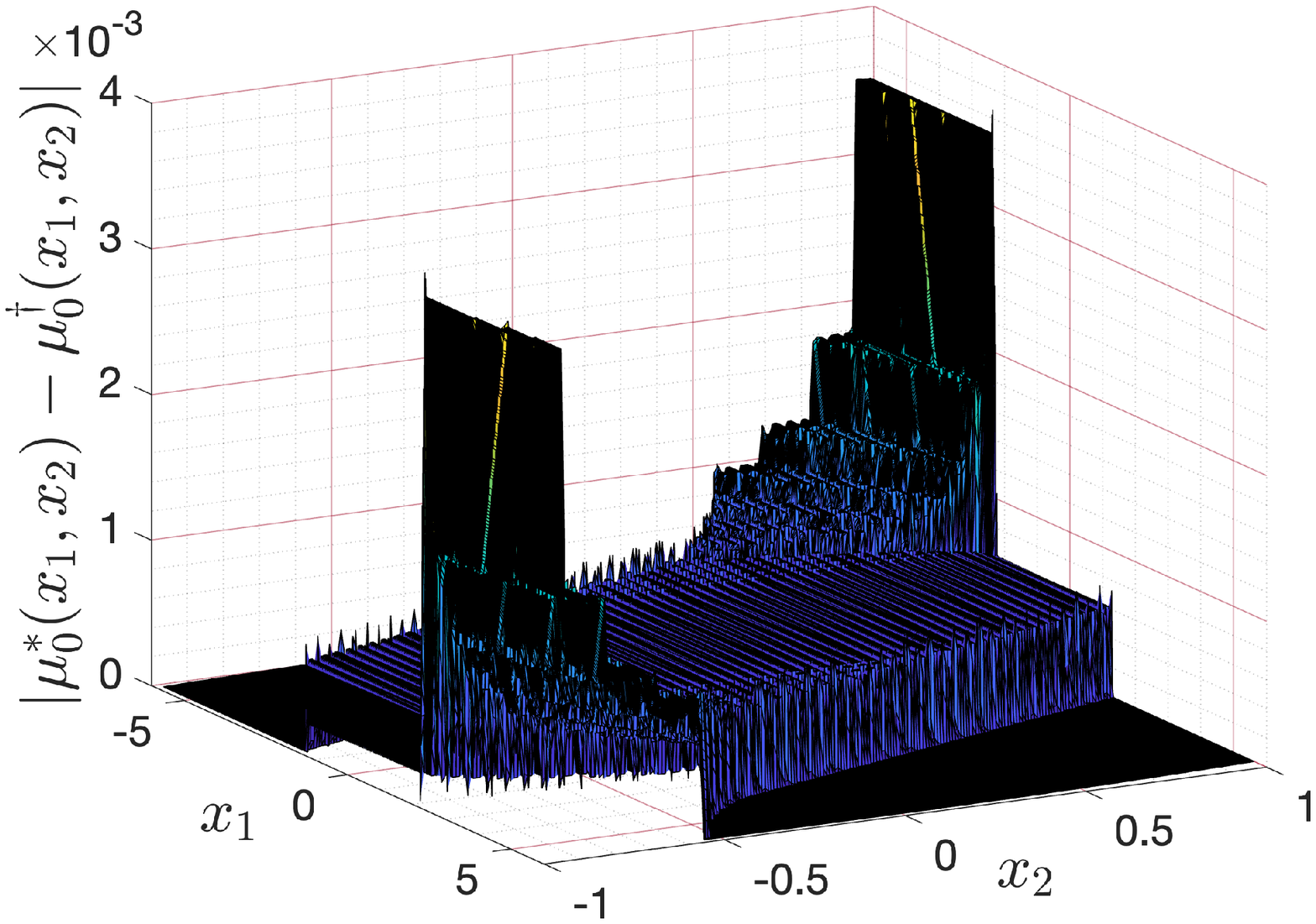}
  \end{subfigure}
	\begin{subfigure}[b]{0.45\textwidth}
    \includegraphics[scale=.30]{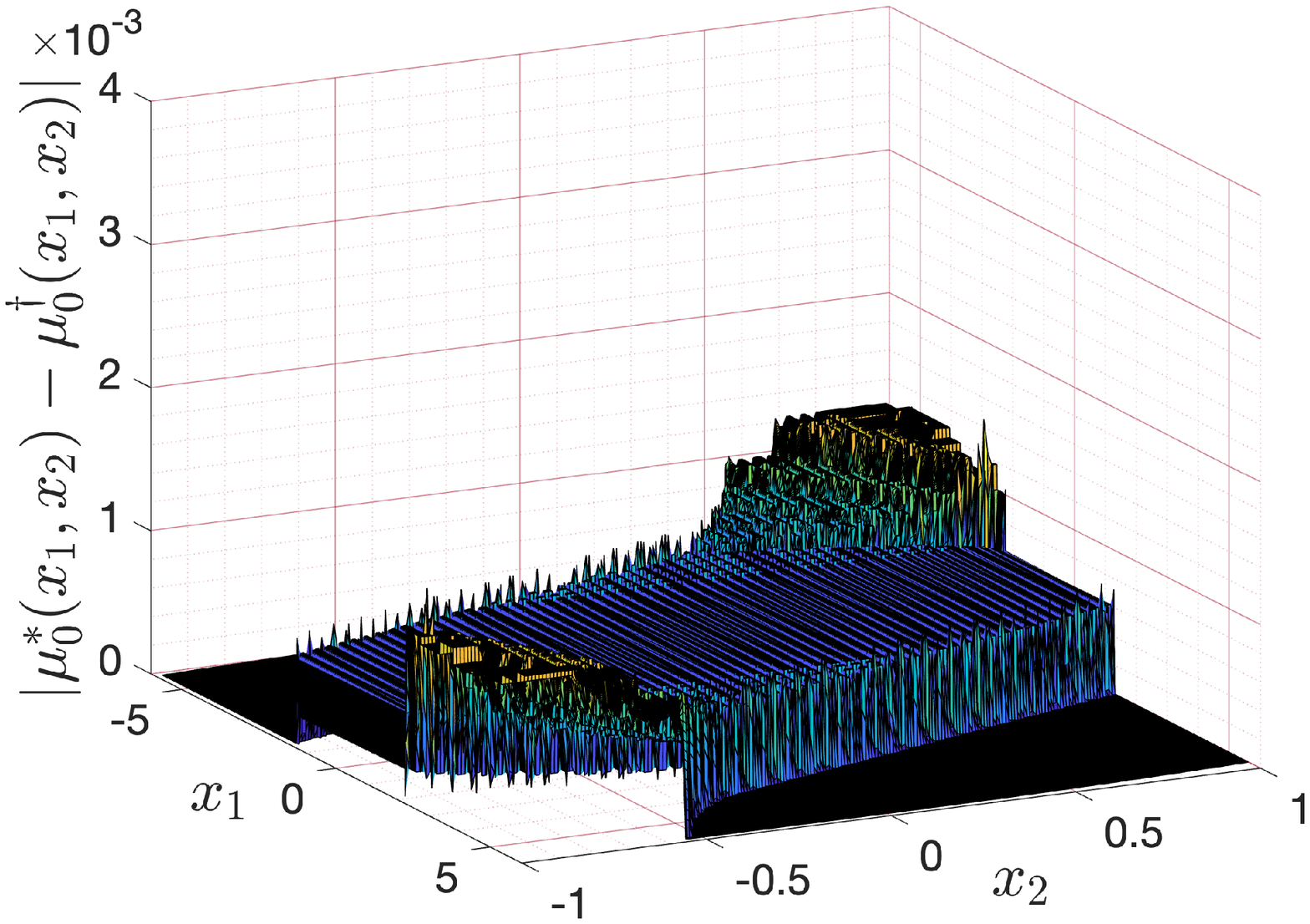}
  \end{subfigure}
\caption{The absolute error between \(\upopt(\cdot)\) and \(\mutrunc(\cdot)\) in Example \ref{exmp:lmpc_1} without the employment of the Lipschitz extension algorithm on the left and with it on the right. Notice that the vertical axis is scaled by the factor of \(10^{-3}\). While in the first case the preassigned error margin \(\eps\) was respected, this may not be typical, and the extension procedure should be carried out in order to conform to the theoretical guarantees.}\label{fig:error_lmpc_ex_1}
\end{figure*}
 \begin{table}[b]
  \centering
  \begin{tabular}{lccc}
        \toprule
          Threshold (\(\eps\))  & \(h\) & \(\mathcal{D}\)                 \\
        \midrule
           \( 50 \times 10^{-3} \) & \(0.04\)  & \(2\)                   \\
          \( 5 \times 10^{-3} \) & \(0.004\)  & \(2\)                   \\
        \bottomrule
    \end{tabular}
    \caption{The error-margin \(\eps\) and the pair \(\big(h,\mathcal{D}\bigr)\) associated with Example \eqref{exmp:lmpc_1}.}
  \label{tab:error_margin_LMPC_ex_1}
\end{table}

\begin{table}[b]
  \centering
\begin{tabular}{lccc}
\toprule
Method  & \textbf{CT} \(\bigl(\mutrunc(\cdot)\bigr)\) & \textbf{CT} (online) & Storage 
\\ \midrule
QuIFS & 23 sec  & 0.5 m.sec & 25 KB
\\
MPT \cite{ref:MPT} &  19 sec & 0.8 m.sec & 17 KB
\\ \bottomrule
    \end{tabular}
    \caption{Computation-time \textbf{CT} and storage requirements with \(\eps=0.05\) for Example \ref{exmp:lmpc_1}.}
  \label{tab:metadata_ex_1}
\end{table}
\end{example}

\begin{example}\label{exmp:lmpc_0}
Consider the fourth-order linear dynamical system:
\begin{align}\label{eq:fourth_order_dyn_II}
    x_{t+1} = Ax_t + B u_t + w_t,
\end{align}
with the state and the actuation matrices:
\[A \Let \begin{pmatrix} 0.4035 & 0.3704 & 0.2935 & -0.7258 \\ -0.2114 & 0.6405 & -0.6717 & -0.0420 \\ 0.8368 & 0.0175 & -0.2806 & 0.3808 \\ -0.0724 & 0.6001 & 0.5552 & 0.4919 \end{pmatrix},\]
\[B \Let \begin{pmatrix} 1.6124 & 0.4086 & -1.4512 & -0.6761 \end{pmatrix}^{\top}.\]
Fix \(N=17\) and consider the robust optimal control problem:
\begin{equation}
	\label{eq:RMPC_fourth_order_II}
	\begin{aligned}
		&  \inf_{(\pi_t)_{t=0}^{\horizon-1}} \sup_{{W}} && \sum_{t=0}^{\horizon-1} \inprod{\dummyx_t}{Q\dummyx_t} + \inprod{\dummyu_t}{R\dummyu_t} \\
		& \sbjto && \begin{cases}
			 \text{the dynamics }\eqref{eq:fourth_order_dyn_II},\,\dummyx_0=\xz,\\
			\dummyx_t \in \Mbb,\text{ and }\dummyu_t \in \Ubb\\
           \quad \text{for all }(w_t,t) \in [-0.01,0.01] \times \timestamp{0}{\horizon-1},\\
            {W} \Let (w_0,\ldots,w_{\horizon-1}),
		\end{cases}
	\end{aligned}
\end{equation}
where \(\Mbb \Let \aset[]{\dummyx \in \Rbb^2 \suchthat \|\dummyx\|_{\infty} \le 5}\) and \(\Ubb \Let \aset[]{\dummyu \suchthat |\dummyu| \le 0.2}\), \(Q \Let I_{4 \times 4}\) is a \(4 \times 4\)-identity matrix, \(R \Let 0.2\). To find the explicit control law we employed the MPT Toolbox \cite{ref:MPT} with \(w_t=0\), which terminated unsuccessfully around \(5\times 10^4\) regions and was not able to compute the explicit feedback. Keeping the disturbance element \(w_t\) as above and applying the explicit synthesis algorithm reported in \cite{ref:DePena_Explicit_minmax} and \cite{ref:Gao:EMPC}, we observed that the algorithm terminated unsuccessfully as well without being able to compute the explicit feedback due to a large number of regions and the corresponding vertex enumeration procedure.

\begin{figure*}[htbp]
	\centering
  \begin{subfigure}[b]{0.45\textwidth}
    \includegraphics[scale=0.32]{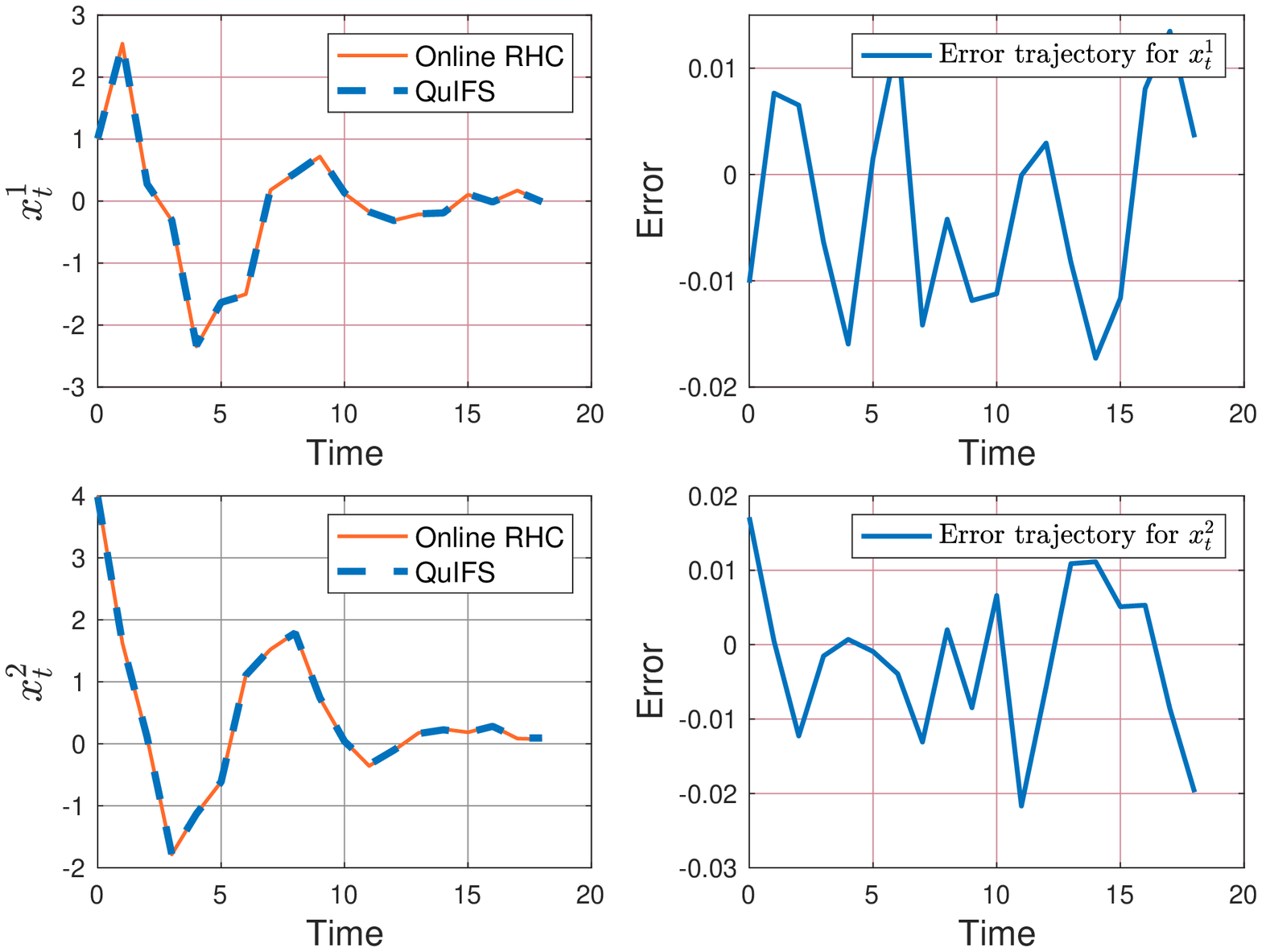}
  \end{subfigure}
  \begin{subfigure}[b]{0.45\textwidth}
    \includegraphics[scale=0.32]{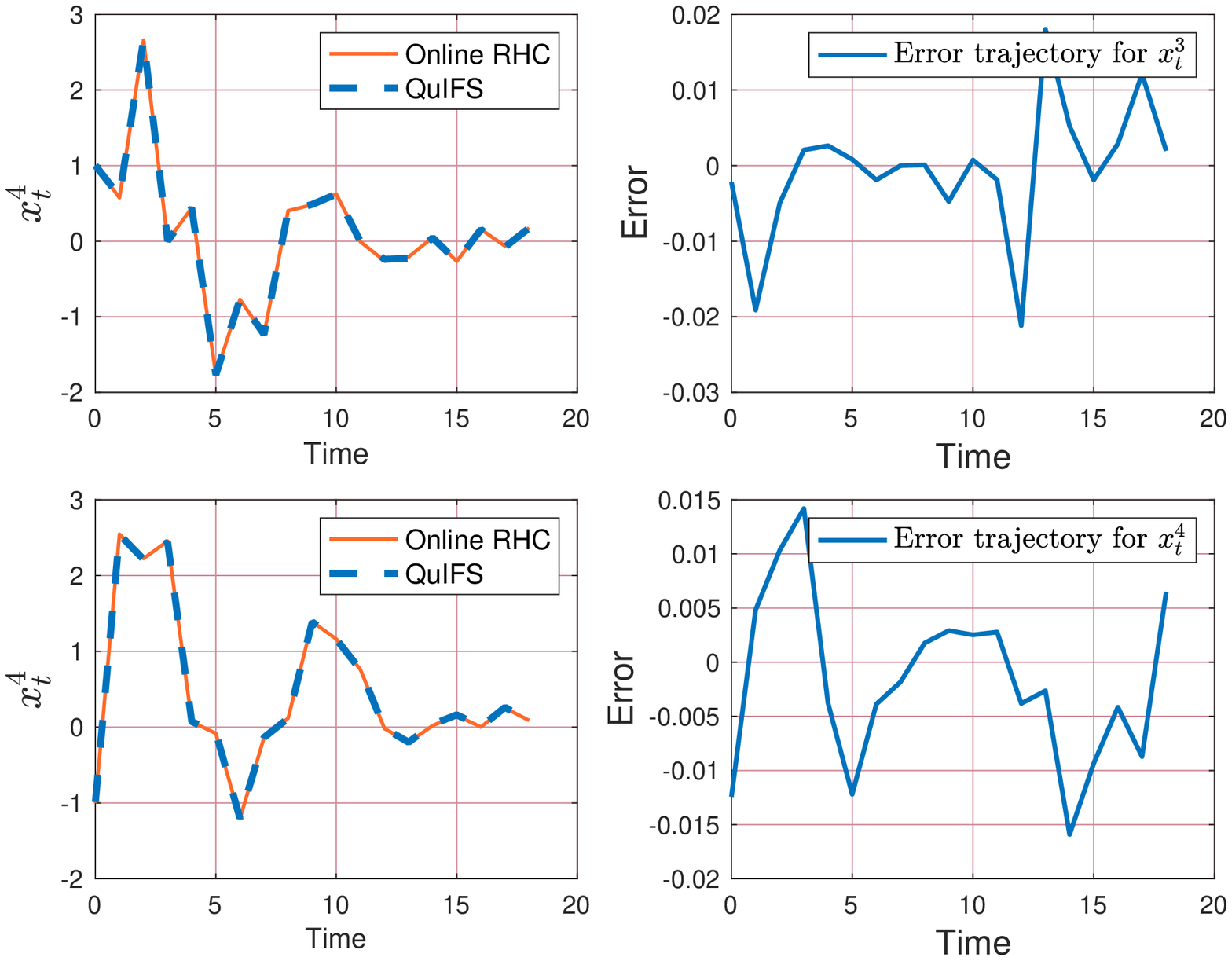}
  \end{subfigure}
\caption{Closed-loop state trajectories starting from \(x(0) \Let (1,4,1,-1)^{\top}\) for Example \ref{exmp:lmpc_0}. QuIFS performs better in terms of closeness to the online RHC trajectories in comparison to the trajectories reported in \cite[Fig. 6(a) and Fig. 6(b)]{ref:summers-multires}; see the state trajectories, and especially the errors in there.}\label{fig:ex_2_fourth_x1_to_x4}
\end{figure*}

We then deployed our QuIFS algorithm. We fix an approximation error \(\eps \Let 0.05\), i.e., \(\|v_t\|_{\infty} \le 0.05\). Define \(\widetilde{w}_t \Let w_t+Bv_t\) and let \(\widetilde{W} \Let \Wbb \oplus B \Vbb\). The ensuing approximation-ready OCP reads as: 
\begin{equation}
	\label{eq:approx_ready_RMPC_fourth_order_II}
	\begin{aligned}
		&  \inf_{(\pi_t)_{t=0}^{\horizon-1}} \sup_{\mathsf{W}} && \sum_{t=0}^{\horizon-1} \inprod{\dummyx_t}{Q\dummyx_t} + \inprod{\dummyu_t}{R\dummyu_t} \\
		& \sbjto && \begin{cases}
	\dummyx_{t+1}=A\dummyx_t+B(\dummyu_t+\dummyv_t)+\dummyw_t,\,\dummyx_0=\xz,\\
			\dummyx_t \in \Mbb,\text{ and }\dummyu_t + \dummyv_t \in \Ubb \\
          \quad \text{for all }(\widetilde{w}_t,t) \in \widetilde{W} \times  \timestamp{0}{\horizon-1},\\
            \mathsf{W} \Let (\widetilde{w}_0,\ldots,\widetilde{w}_{\horizon-1}).
		\end{cases}
	\end{aligned}
\end{equation}
We kept the same computer/server specifications as in Example \eqref{exmp:lmpc_1}, and employed the solver MOSEK along with YALMIP's robust optimization framework to solve the problem \eqref{eq:approx_ready_RMPC_fourth_order_II} using disturbance feedback parameterization of the control policy \cite[Eq. 12a]{ref:Lof:minmax:cdc} at grid points of the underlying state space, with the grid size \(h = 0.01\) (specified below) dictated by the QuIFS algorithm and per point computation time was \(\sim\) 5 sec.
For us the dimension of the state-space \(d=4\). Selecting \(M_0=3\) and using \eqref{num:lagueere_pol}-\eqref{num:high_order_basis}, we get \(\mathsf{L}_2^2 (\dummyx) \Let 6 - 4 \dummyx + \tfrac{\dummyx^2}{2}\) and consequently 
\begin{align}\label{num:sixth_order_basis_nmpc}
      \hspace{-2mm} \Rbb^4 \ni  x \mapsto \psi_6(x) \Let \frac{1}{\pi^2} \left( 6- 4 \|x\|^2+ \frac{1}{2}\|x\|^4\right)\epower{-\|x\|^2}.
    \end{align}
The generating function \eqref{num:sixth_order_basis_nmpc} satisfies the moment condition of order \(M=6\), which means that the constant \(C_{\gamma} = 1/7\). Next, we fix the shape parameter \(\Dd = 2\), and we obtain the value of \(h = \frac{\varepsilon}{3C_\gamma L_0\sqrt{\Dd}} \approx 0.01\), where we have employed a conservative estimate of the Lipschitz constant \(L_0 = 8\) of \(\upopt(\cdot)\). Define \(\overline{z} \Let x-mh\). We picked \(\rzero=5\), which means that \(11\) terms are used in the quasi-interpolation formula:
\begin{align}\label{num:sixth_quasi_interpolant}
     \mutrunc(x) \Let\hspace{-1mm} \frac{1}{(\pi\Dd)^2}\hspace{-1mm}\sum_{mh \in \finset_x(\rzero)}\hspace{-4mm}\upopt(mh) \biggl( 6 - 4 \frac{\|\overline{z}\|^2}{h^2\Dd}+\frac{1}{2}\frac{\|\overline{z}\|^4}{h^4 \Dd^2}\biggr) \epower{-\frac{\|\overline{z}\|^2}{\Dd h^2}}
 \end{align}
for the approximate feedback synthesis. Consequently, it is guaranteed that with the approximant \(\mutrunc(\cdot)\) in \eqref{num:sixth_quasi_interpolant} and \((h,\Dd)=(0.01,2)\) the estimate \(\|\upopt(\cdot)-\mutrunc(\cdot)\|_{\mathrm{u}} \le 0.05\) holds.
We compare our results with the online receding horizon control trajectory. Observe that our results are visibly better than the one reported in \cite[Fig. 6(c)]{ref:summers-multires} in terms of closeness between the online RHC and the approximate trajectory and the error characteristics.
The storage requirements, the computation-time to generate the feedback map \(\mutrunc(\cdot)\), and the online computation-time are recorded in Table \eqref{tab:metadata_ex_0} for horizon \(N=17\). Similar statistics with \(N=6\) are given in Table \ref{tab:metadata_new_ex_0}. Figure \eqref{fig:ex_2_fourth_x1_to_x4} depicts the state trajectories and Figure \eqref{fig:ex_2_fourth_u} shows the online receding horizon and the approximate (explicit) control trajectories along with the error between them obtained from QuIFS on the same time scale.
\begin{figure}[htbp]
	\centering
\includegraphics[scale=0.4]{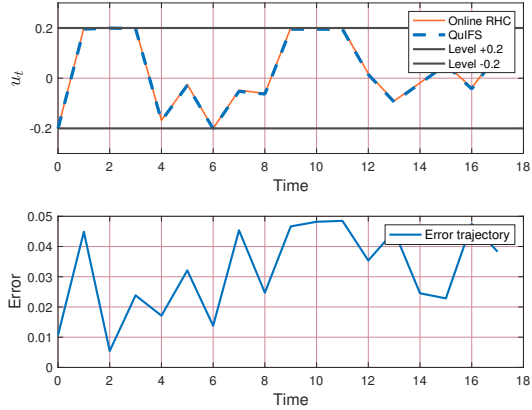}
\caption{The online receding horizon control and the solution obtained from QuIFS for Example \ref{exmp:lmpc_0}. See \cite[Fig. 6(c)]{ref:summers-multires} for a comparison.}\label{fig:ex_2_fourth_u}
\end{figure}
\begin{table}[h]
\centering
\begin{tabular}{lccc}
\toprule
Method  & \textbf{CT} \(\bigl(\mutrunc(\cdot)\bigr)\) & \textbf{CT} (online) & Storage \\ 
\midrule
QuIFS &  187 sec & 5 m.sec & 9 MB        
\\
 MPT \cite{ref:MPT} &  Terminated & NA & Terminated \\ 
 \cite{ref:DePena_Explicit_minmax, ref:Gao:EMPC} & unsuccessfully & & unsuccessfully \\
\bottomrule
    \end{tabular}
    \caption{Computation-time \textbf{CT} and storage requirements with horizon \(\horizon=17\) and \(\eps=0.05\) for Example \ref{exmp:lmpc_0}.}
  \label{tab:metadata_ex_0}
\end{table}

\begin{table}[h]
\centering
\begin{tabular}{lccc}
\toprule
Method  & \textbf{CT} \(\bigl(\mutrunc(\cdot)\bigr)\) & \textbf{CT} (online) & Storage \\ 
\midrule
QuIFS &  78 sec & 1 m.sec & 0.9 MB \\
\cite{ref:DePena_Explicit_minmax} & 120 sec  & 10 m.sec & 0.5 MB \\ 
 \cite{ref:Gao:EMPC} & 90 sec &  2 m.sec & 0.7 MB  \\
\bottomrule
\end{tabular}
\caption{Computation-time \textbf{CT} and storage requirements with horizon \(\horizon=6\) and \(\eps=0.05\) for Example \ref{exmp:lmpc_0}.}
\label{tab:metadata_new_ex_0}
\end{table}
\end{example}


\subsection{Nonlinear MPC}\label{sec:num_NMPC} In this section we provide two numerical examples concerning \emph{nonlinear} MPC to demonstrate the effectiveness of our algorithm. 
\begin{example}\label{exmp:nmpc_3}
Consider the continuous-time second-order nonlinear controlled dynamical system \cite{ref:canale:sm:approx}: 
\begin{align}\label{eq:nonlinear_3_example}
    \dot{x}_1(t)& = x_2(t) \nn \\ \dot{x}_2(t) & = u(t) - 0.6 x_2(t) -x_1(t)^3 - x_1(t).
\end{align}
\begin{figure*}[h]
	\centering
  \begin{subfigure}[b]{0.32\linewidth}
    \includegraphics[scale=0.22]{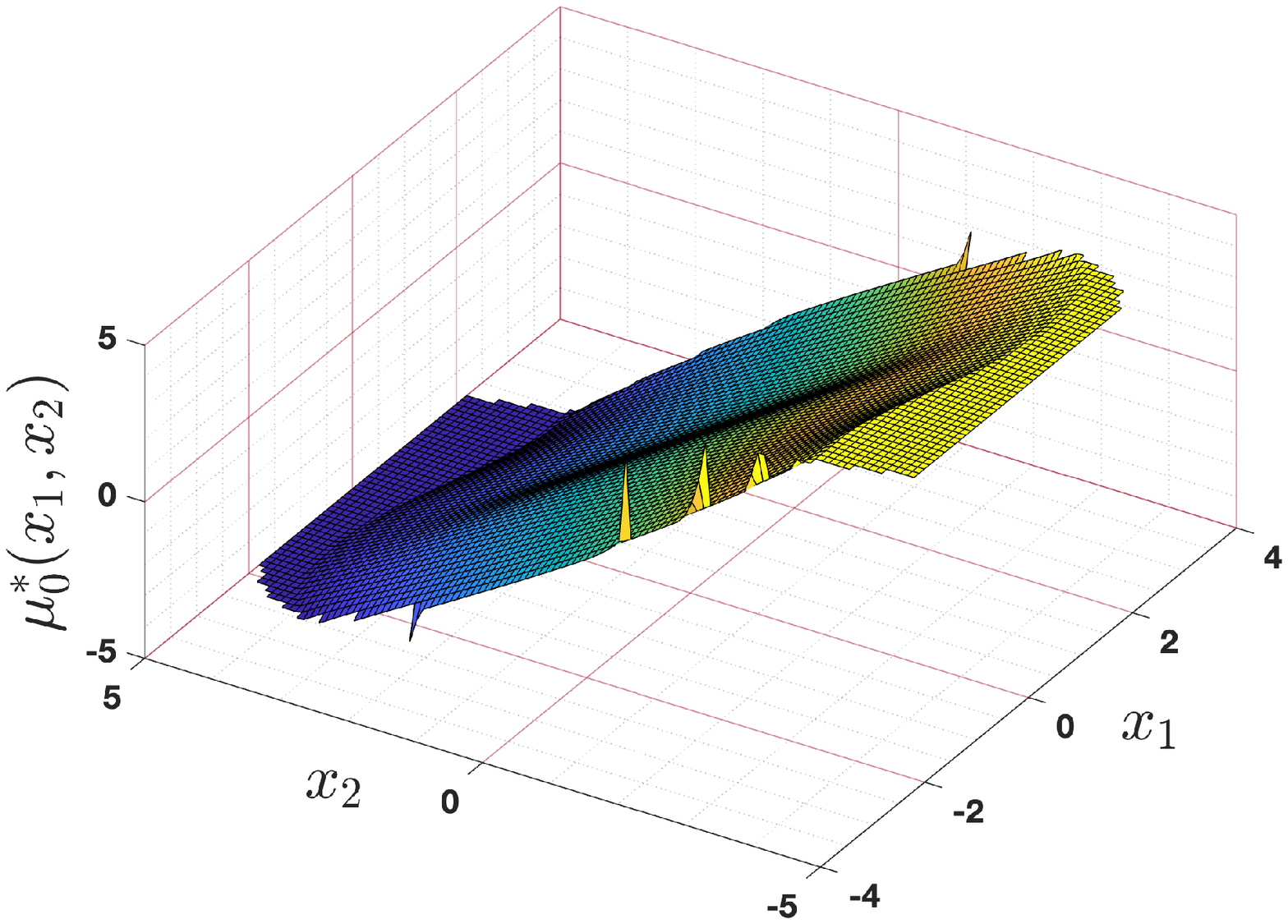}
  \end{subfigure}
  \begin{subfigure}[b]{0.32\linewidth}
    \includegraphics[scale=0.22]{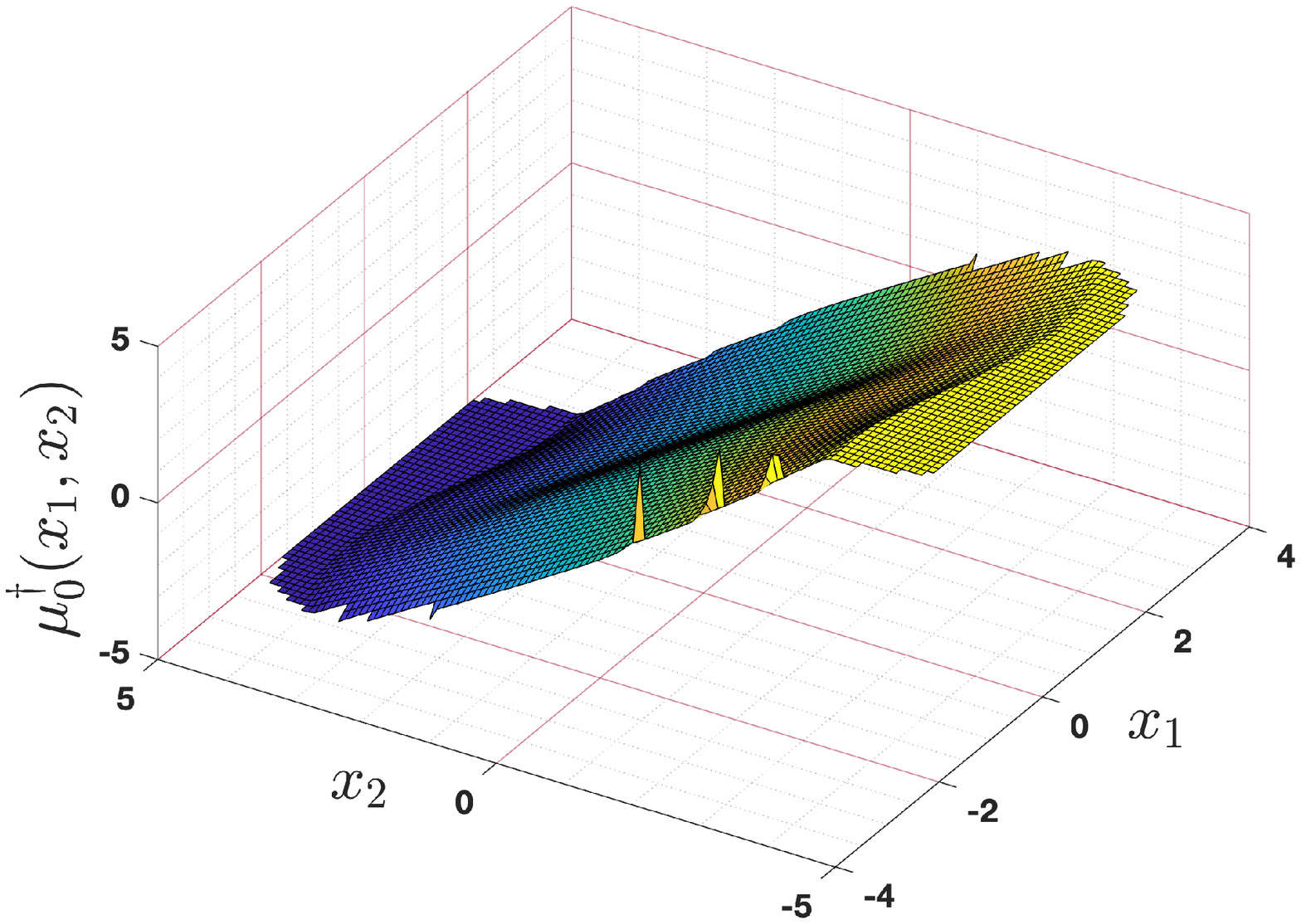}
  \end{subfigure}
  \begin{subfigure}[b]{0.3\linewidth}
    \includegraphics[scale=0.22]{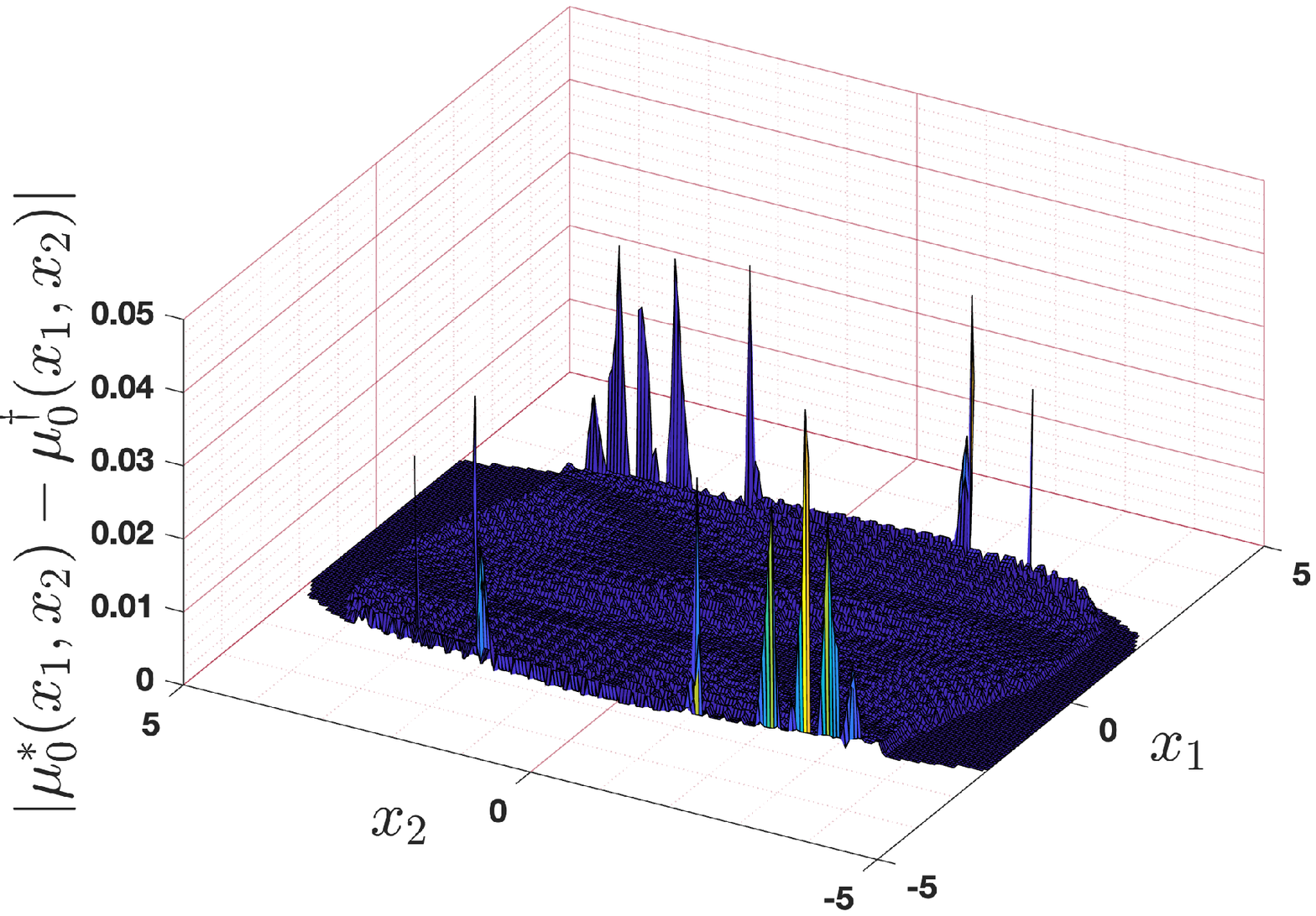}
  \end{subfigure}
\caption{The feedback \(\upopt(\cdot)\), the explicit feedback \(\mutrunc(\cdot)\), and the absolute error between \(\upopt(\cdot)\) and \(\mutrunc(\cdot)\) for Example \eqref{exmp:nmpc_3} with \(\eps=0.05\), and \(\bigl(h,\Dd \bigr)=(0.01,2)\).}\label{fig:can_Actual_Approx_NMPC}
\end{figure*}
The system dynamics \eqref{eq:nonlinear_3_example} is discretized using a forward-Euler scheme with sampling time \(T_s \Let 0.05\). Fix \(\horizon \Let 100\), and consider the finite horizon discrete-time optimal control problem 
\begin{align}
        \label{eq:exmp_3_nonlinear-original problem}
        \begin{aligned}
            &\inf_{(\dummyu_t)_{t=0}^{\horizon-1}}  && \hspace{-2mm}\sum_{t=0}^{\horizon-1} \inprod{\dummyx_t}{Q\dummyx_t}+\inprod{\dummyu_t}{R\dummyu_t}\\
            &\sbjto && \hspace{-2mm} \begin{cases}
          \text{discretized dynamics }\eqref{eq:nonlinear_3_example}, \,\dummyx_0 = \xz, \\
          \dummyx_t \in \Mbb, \text{ and }\dummyu_t \in \Ubb
            \text{ for all }t \in \timestamp{0}{\horizon-1},
            \end{cases}
        \end{aligned}
    \end{align}
where \(\Mbb \Let \aset[]{\dummyx \in \Rbb^2 \suchthat \|\dummyx\|_{\infty} \le 5}\), \(\Ubb \Let \aset[]{\dummyu \in \Rbb \suchthat |\dummyu| \le 5}\), \(Q \Let \mathbb{I}_{2 \times 2}\), and \(R \Let 0.5\). Let us fix \(\eps = 0.05 \). For the synthesis of the approximation-ready policy, as per our formulation, we consider the dynamics: 
\begin{align}\label{eq:nonlinear_3_example_noisy}
    \dot{x}_1(t) &= x_2(t), \nn \\
        \dot{x}_2(t) &= u(t) + v(t) -0.6x_2(t)-x_1(t)^3-x_1(t).
\end{align}
Then the approximation-ready robust optimal control problem is 
\begin{equation}
	\label{eq:approx_ready_RMPC_num_ex_0}
	\begin{aligned}
		&  \inf_{(\dummyu_t)_{t=0}^{\horizon-1}} \sup_{\mathsf{W}} && \sum_{t=0}^{\horizon-1} \inprod{\dummyx_t}{Q\dummyx_t} + \inprod{\dummyu_t}{R\dummyu_t} \\
		& \sbjto && \begin{cases}
			 \text{discretized dynamics }\eqref{eq:nonlinear_3_example_noisy},\,\dummyx_0= \xz,\\
			\dummyx_t \in \Mbb, \text{ and }\dummyu_t + \dummyv_t \in \Ubb \\
           \quad\text{for all } (\dummyv_t,t)\in [-0.05,0.05] \times \timestamp{0}{\horizon-1},\\
            \mathsf{W} \Let (\dummyv_0,\ldots,\dummyv_{\horizon-1}).
		\end{cases}
	\end{aligned}
\end{equation}
The solution to the problem \eqref{eq:approx_ready_RMPC_num_ex_0} was obtained by gridding the state-space \([-5,5] \times [-5,5] \) with step size \(h=0.01\) and solving the ensuing NLP \eqref{eq:approx_ready_RMPC_num_ex_0} in YALMIP using the NLP solver \(\mathsf{IPOPT}\) \cite{ref:IPOPT_Biegler} deriving a robust counterpart. The problem \eqref{eq:approx_ready_RMPC_num_ex_0} can be alternatively solved using the algorithm reported in \cite{ref:JK:MM:FA:2021} as well. We performed our numerical computation keeping the computer server identical to the one in Example \eqref{exmp:lmpc_1}, and computation-time per point was \(\sim\) \(2.3\) sec. Next we choose the Laguerre-Gaussian basis function with \(d=2\), and \(M_0=3\):
\begin{align}\label{num:sixth_order_basis}
      \Rbb^2 \ni  x \mapsto \overline{\psi}_6(x) \Let \frac{1}{\pi} \left( 3- 3 \|x\|^2+ \frac{1}{2}\|x\|^4\right)\epower{-\|x\|^2}.
    \end{align}
The generating function \(\overline{\psi}_6(\cdot)\) satisfies a moment condition of order \(M=6\), which means that the constant \(C_{\gamma} = 1/7\). Next, we fix the shape parameter \(\Dd = 2\), we obtain the value of \(h = \frac{\varepsilon}{3C_\gamma L_0\sqrt{\Dd}} \approx 0.01\), where \(L_0 = 4\) is a rough estimate of Lipschitz constant of \(\upopt(\cdot)\), estimated numerically. Choosing \(\rzero=3\), we employ the quasi-interpolant \eqref{quasi_lmpc_ex_1} with \(\overline{\psi}_{6}(\cdot)\) in place of \(\psi(\cdot)\) to generate the approximate policy \(\mutrunc(\cdot)\). We skipped the Lipschitz extension procedure for this example due to minor error fluctuations in the boundary of the feasible set. Figure \eqref{fig:can_Actual_Approx_NMPC} shows the actual and the approximated policies along with the error surface. Table \eqref{tab:metadata_ex_nl_2} collects the computation-time and storage requirement data concerning the explicit feedback law. 
Figures \eqref{fig:canale_state} and \eqref{fig:canale_control} depict the closed-loop trajectories obtained via applying QuIFS and \cite{Raimondo2012}. Clearly (see the enlarged snippets within the figures) QuIFS performs better than the algorithm established in \cite{Raimondo2012} in terms of closeness to the online RHC trajectories. 
\begin{figure}[ht]
\centering
\includegraphics[scale=0.4]{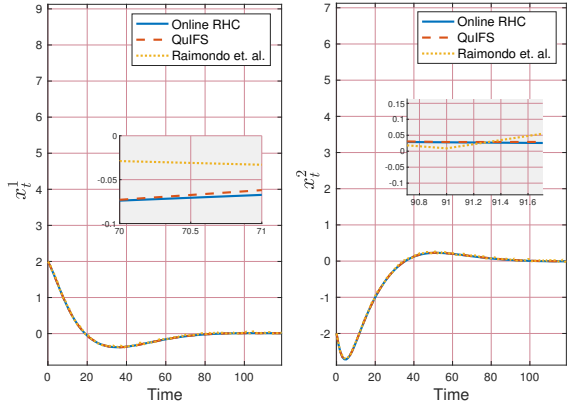}
\caption{State trajectories for Example \ref{exmp:nmpc_3} generated by QuIFS, the algorithm reported in \cite{Raimondo2012}, and the online receding horizon state trajectories for \(x_0 \Let \begin{pmatrix}2&-2 \end{pmatrix}^{\top}\) for a simulation length \(120\).}\label{fig:canale_state}
\end{figure}
\begin{figure}[ht]
\centering
\includegraphics[scale=0.35]{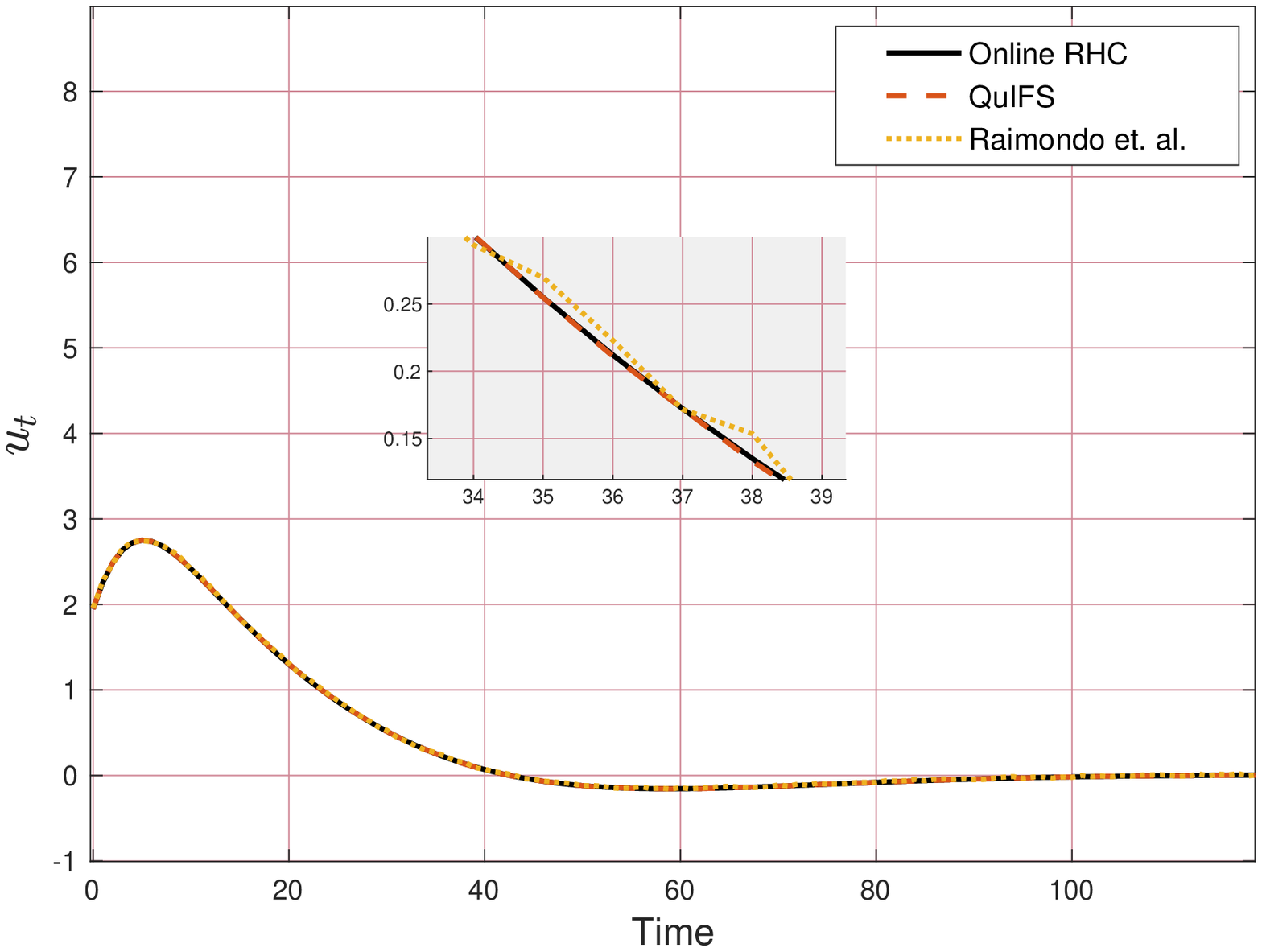}
\caption{Control trajectories for Example \ref{exmp:nmpc_3} generated by QuIFS, the algorithm reported in \cite{Raimondo2012}, and the online receding horizon control trajectory for a simulation length \(120\).}\label{fig:canale_control}
\end{figure}
\begin{table}[]
  \centering
\begin{tabular}{lccc}
\toprule
Method  & \textbf{CT} \(\bigl(\mutrunc(\cdot)\bigr)\) & \textbf{CT} (online) & Storage 
\\ \midrule
QuIFS & 51 sec & 0.8 m.sec & 35 KB   \\
\cite{Raimondo2012} &  45 sec & 0.8 m.sec & 28 KB 
\\ \bottomrule
    \end{tabular}
    \caption{Computation-time \textbf{CT} and storage requirement with \(\eps=0.05\) for Example \ref{exmp:nmpc_3}.}
  \label{tab:metadata_ex_nl_2}
\end{table}
\end{example}


\begin{example}\label{exmp:nmpc_1}
Consider the dynamical system \cite{Raimondo2012}:
\begin{align} \label{eq:nonlinear_example}
   \dot{\st}_1(t)&=\st_2(t)+ \bigl(0.5 +0.5 \st_1(t) \bigr)\ut(t) \nn \\\dot{\st}_2(t) &= \st_1(t)+ \bigl(0.5 -2\st_2(t) \bigr)\ut(t).
\end{align}
We discretize the system dynamics \eqref{eq:nonlinear_example} using a Runge-Kutta (RK4) discretization scheme with sampling time \(T_s \Let 0.1\). For a fixed time horizon \(\horizon \Let 15\), consider the finite horizon discrete-time optimal control problem 
\begin{align}\label{eq:exmp_2_nonlinear-original problem}
        \begin{aligned}
            &\inf_{(\dummyu_t)_{t=0}^{\horizon-1}}  && \hspace{-2mm}\inprod{\dummyx_N}{P\dummyx_N}+\sum_{t=0}^{\horizon-1}\inprod{\dummyx_t}{Q\dummyx_t}+\inprod{\dummyu_t}{R\dummyu_t}\\
            &\sbjto && \hspace{-2mm}\begin{cases}
          \text{discretized dynamics }\eqref{eq:nonlinear_example},\,\dummyx_0= \xz,\\
            \dummyx_t \in \Mbb,\,\dummyx_{\horizon} \in \Mbb_F \text{ and }\dummyu \in \Ubb\text{ for all }t \in \timestamp{0}{\horizon-1},
            \end{cases}
        \end{aligned}
    \end{align}
where \(\Mbb \Let \aset[]{\dummyx \in \Rbb^2 \suchthat \|\dummyx\|_{\infty} \le 1}\), \(\Ubb \Let \aset[]{\dummyu \in \Rbb \suchthat |\dummyu| \le 1}\), \(Q \Let 0.01 I_{2\times 2}\), \(R\Let 0.01\), and \(P \Let \begin{psmallmatrix}19.6415&13.1099 \\13.1099&19.6414
\end{psmallmatrix}\), which is found by solving the Lyapunov equations \cite{Raimondo2012}. The terminal region is \(\Mbb_{F} \Let \aset[\big]{\st \in \Rbb^2 \mid \inprod{\st}{P\st} \le 1}.\) The solution to \eqref{eq:exmp_2_nonlinear-original problem} is obtained by gridding the state-space \([-1,1] \times [-1,1] \) with step size \(h=0.0035\) (dictated by the QuIFS algorithm). We employed the NLP solver \(\mathsf{IPOPT}\) in YALMIP with the same computer specifications as in Example \eqref{exmp:lmpc_1} to solve the optimization problem \eqref{eq:exmp_2_nonlinear-original problem}, where per point computation-time was \(\sim\) \(3\) sec. The Lipschitz constant of \(\upopt(\cdot)\) is roughly \(L_{0} \approx 2\) as estimated by numerical techniques. For the quasi-interpolation scheme we choose the Laguerre generating function by fixing \(M_0=5\) and \(d=2\) in \eqref{num:high_order_basis}
\begin{align}\label{num:tenth_order_basis}
      & \Rbb^2 \ni x \mapsto \psi_{10}(x) \Let \frac{1}{\pi} \big{(} 5- 10\|x\|^2 +5\|x\|^4-\frac{5}{6}\|x\|^6  \nn \\& +\frac{1}{24}\|x\|^8\big{)}\epower{-\|x\|^2}.
\end{align}
which satisfies moment condition of order \(M=10\), and consequently \(C_{\gamma} = 1/11\). For illustration, let us fix a error tolerance \(\eps= 0.005\), and fix the shape parameter \(\Dd = 2\). Simple algebra leads to the parameter \(h = \frac{\eps}{3C_{\gamma}L_{0}\sqrt{\Dd}} = 0.0063\). Observe that the value of \(h\) is conservative in nature, and it may be possible to achieve the error tolerance with higher values of \(h\). We choose \(R_0=7\), and employ the quasi-interpolant \eqref{quasi_lmpc_ex_1} with \(\psi_{10}(\cdot)\) in place of \(\psi(\cdot)\) to generate the approximate policy \(\mutrunc(\cdot)\) and it is guaranteed that the pair \(\bigl(h,\Dd\bigr)=\bigl(0.0063,2\bigr)\) leads to \(\|\upopt(\cdot)-\mutrunc(\cdot)\|_{\mathrm{u}} \le 0.005\). Figure \eqref{fig:error_nmpc_ex_2} numerically verifies this fact. 
Corresponding to the error margin \(\eps=0.005\), the error surface is shown in Figure \eqref{fig:error_nmpc_ex_2}. Table \eqref{tab:metadata_nl_ex_1} records computation-time and storage requirements corresponding to an error margin \(\eps = 0.05\).

\begin{figure}[t]
\includegraphics[scale=0.4]{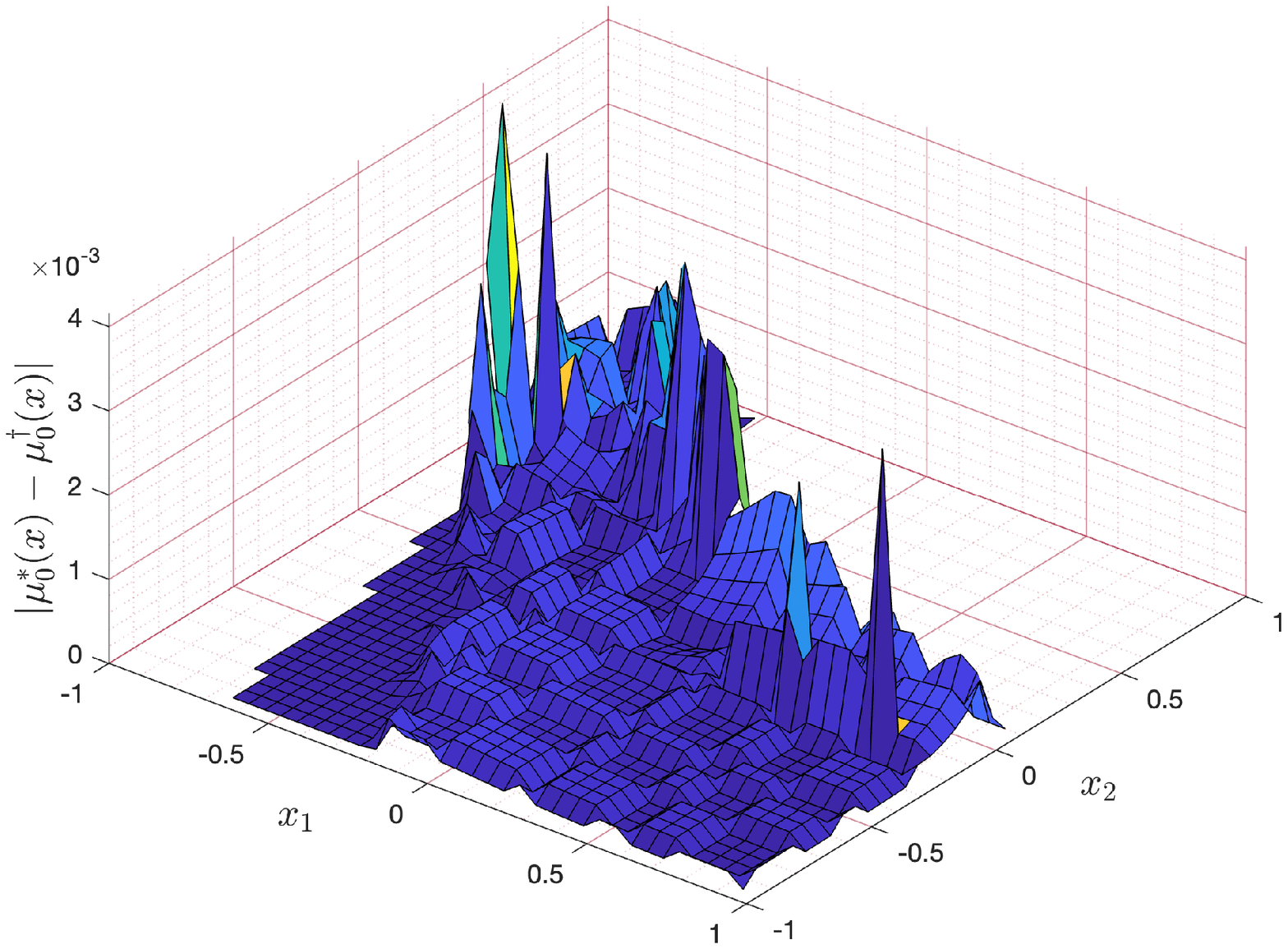}
\caption{The absolute error between \(\upopt(\cdot)\) and \(\mutrunc(\cdot)\) for Example \eqref{exmp:nmpc_1}.}\label{fig:error_nmpc_ex_2}
\end{figure}
\begin{figure}[t]
\centering
\includegraphics[scale=0.45]{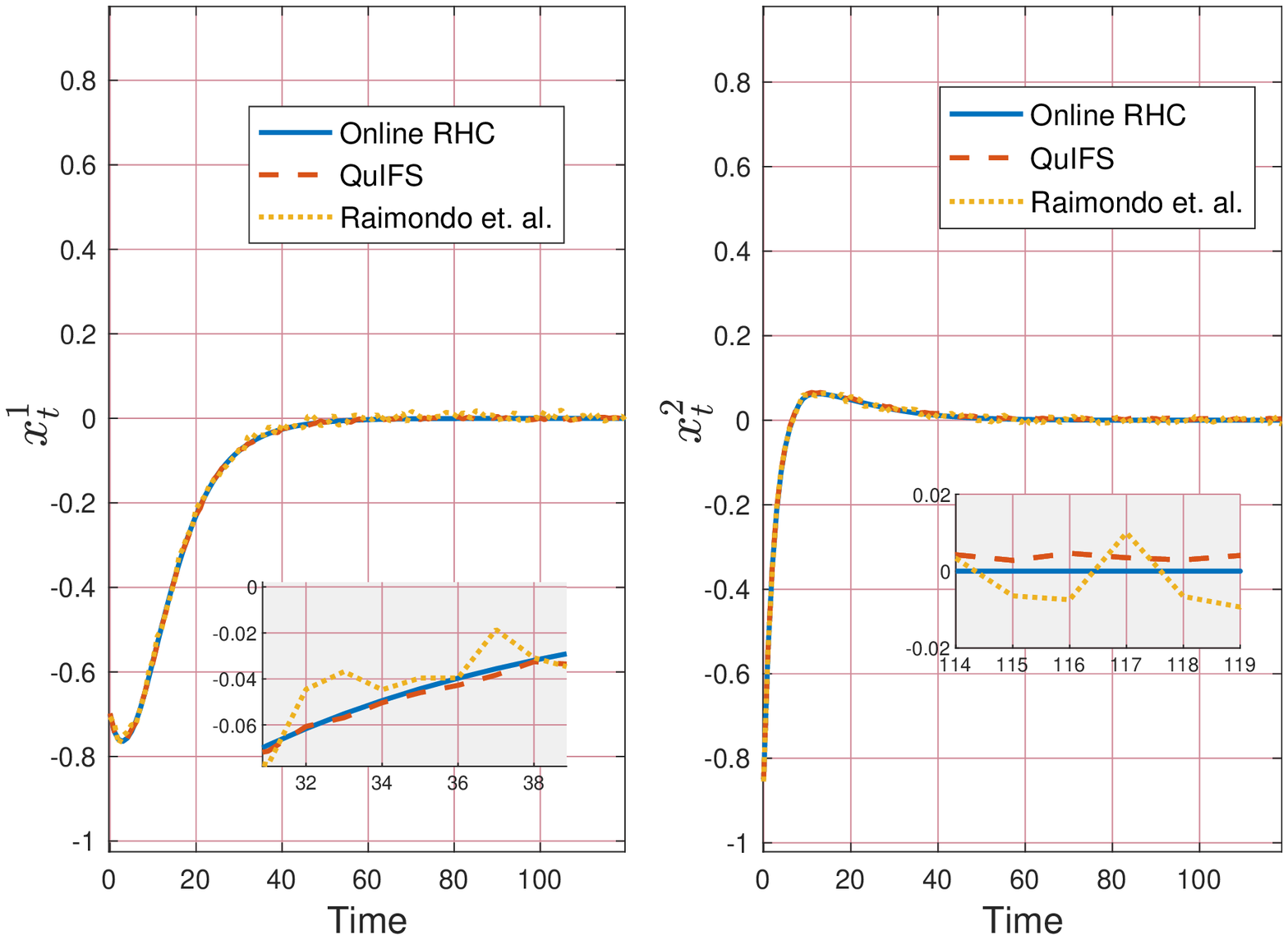}
\caption{State trajectories for Example \eqref{exmp:nmpc_1} generated by QuIFS, the algorithm reported in \cite{Raimondo2012}, and the online receding horizon state trajectories for \(x_0 \Let \begin{pmatrix}-0.7&-0.85 \end{pmatrix}^{\top}\) for a simulation length \(120\).}\label{fig:raimondo_state}
\end{figure}
\begin{figure}[t]
\centering
\includegraphics[scale=0.4]{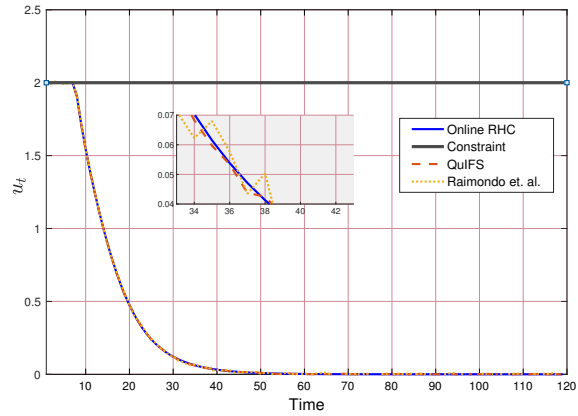}
\caption{Control trajectories for Example \eqref{exmp:nmpc_1} generated by QuIFS, the algorithm reported in \cite{Raimondo2012}, and the online receding horizon control trajectory for a simulation length \(120\). }\label{fig:raimondo_control}
\end{figure}
Moreover, it can be seen (see the enlarged snippets within the figures) that in terms of closeness to the online RHC trajectories QuIFS does a better job compared to the hierarchical grid-based technique reported in \cite{Raimondo2012}. Indeed, there is visibly less oscillatory behaviour under QuIFS compared to \cite{Raimondo2012}. 
\begin{table}[!h]
  \centering
\begin{tabular}{lccc}
\toprule
Method  & \textbf{CT} \(\bigl(\mutrunc(\cdot)\bigr)\) & \textbf{CT} (online) & Storage 
\\ \midrule
QuIFS & 77 sec  & 0.4 m.sec & 34 KB \\
\cite{Raimondo2012} & 61 sec & 0.7 m.sec & 30 KB
\\ \bottomrule
    \end{tabular}
    \caption{Computation-time \textbf{CT} and storage requirement with \(\eps=0.05\) for Example \eqref{exmp:nmpc_1}.}
  \label{tab:metadata_nl_ex_1}
\end{table}
\end{example}


%% file: verArXiVconcl.tex
\section{Concluding remarks}

We introduced QuIFS, a one-shot approximate feedback synthesis algorithm for nonlinear robust MPC problems, and we theoretically established guarantees for uniform error estimates between the optimal and the approximate feedback policies. It was also shown that under the explicit approximate policy, the closed-loop process is ISS-like stable and the ensuing optimization problem is recursively feasible under standard mild hypotheses. One of the future directions of this work involves employing techniques from deep learning to synthesize approximate feedback laws with one-shot and uniform guarantees of convergence and compare them with QuIFS. Natural extensions to the case of stochastic MPC are being developed and will be reported separately.

%% file: verArXiVapp.tex
\section{Appendix B: Stability notions}\label{sec:appenb}

This section collects several definitions and results surrounding input to state stability of discrete-time controlled systems. The theory of input-to-state stability (ISS) and many of its variants have been extensively employed in the analysis of robust stability properties for both continuous-time \cite{ref:Sontag-95} and discrete-time \cite{ref:JiaWang-2001} nonlinear dynamical systems subjected to uncertainty. We recollect a few results on discrete-time regional ISS stability that are commonly employed in MPC:

The basic object under consideration is a discrete-time dynamical system
\begin{align}\label{appenb:sys}
    x_{t+1}= \field (x_t,w_t), \quad x_0=\xz\,\,\text{given},\,\,t \in \Nz,
\end{align}
where \(x_t \in \Rbb^d\) is the vector of states and \(w_t \in \Wbb \subset \Rbb^p\) is the input disturbance vector, \(\Wbb\) is a compact set and \(\field(0,0)=0\). By \(\st\bigl(t;\xz,w \bigr)\) we denote the state trajectory of the system \eqref{appenb:sys} with initial state \(\xz\) and input sequence \(w = (w_t)_{t\in\Nz}\). 

\begin{definition}\label{appen_b_RPI}\cite[Definition 4]{ref:LG:DMR:RS:RISS}
	A set \(\mathcal{X} \subset \Rbb^d\) is a \emph{robust positively invariant} set (RPI) for the system \eqref{appenb:sys} if \(f(x,w) \in \mathcal{X}\) for all \(x \in \mathcal{X}\) and for all \(w \in \Wbb\).
\end{definition}

\begin{definition}\label{appen_b_ISS_Lyap_Magni}\cite[Definition 8]{ref:LG:DMR:RS:RISS}
	Consider the system \eqref{appenb:sys}, and suppose that \(\mathcal{X} \subset \Rbb^d\) is compact robustly positive invariant set, and that \(\mathcal{Y}\) and \(\mathcal{Z}\) are compact sets containing the origin as an interior point and satisfying \(\mathcal{Z} \subset \mathcal{Y}\subset \mathcal{X}\). A function \(V:\Rbb^d \lra \lcro{0}{+\infty}\) is a \emph{regional ISS Lyapunov function} in \(\mathcal{X}\) if:
	\begin{itemize}[leftmargin=*, label=\(\triangleright\)] 
    \item there exist \(\classKinfty\) functions \(\alpha_1,\alpha_2\), and \(\alpha_3\) and a \(\classK\) function \(\sigma\) such that the following inequalities hold for all \(w \in \Wbb\):
\begin{align*}
    V(\dummyx) &\ge \alpha_1(|\dummyx|) \quad\,\text{for all}\, \,\dummyx\in \mathcal{X}, \nn \\ V(\dummyx) & \le \alpha_2(|\dummyx|)\quad \,\text{for all}\,\,\dummyx\in \mathcal{Y}, \nn \\ V\circ f\bigl(\dummyx,w) - V(\dummyx) & \le -\alpha_3(|\dummyx|)+\sigma(|w|)\quad\,\text{for all}\,\,\dummyx \in \mathcal{X};
\end{align*}
\item there exists a \(\classKinfty\) function \(\rho\) with the property that \((\mathrm{id}-\rho) \in \classKinfty\), \(\alpha_4 \Let \alpha_3 \circ \alpha_2^{-1}\), and \(b \Let \alpha_4^{-1} \circ \rho^{-1} \circ \sigma\) such that \(\mathcal{Z}\) can be defined for some arbitrary \(c>0\) in the following way:
\[
    \mathcal{Z} \Let \aset[\big]{\dummyx \in\Rbb^d \mid d(\dummyx,\partial \mathcal{Y})>c,\, V(\dummyx) \le b\bigl(\|w\|_{\infty}\bigr)},
\]
			where \(d(\dummyx, \partial\mathcal{Y})\) is the distance of \(\dummyx\) from the boundary of the set \(\mathcal Y\).
\end{itemize}
\end{definition}

\begin{definition}[{\cite[Definition A.3]{ref:Pin-parsini-2012-IJC-di}}]
	\label{appen_b_Regional_Practical_ISS}
	Given a compact set \(\mathcal{X} \subset\Rbb^d\), if \(\mathcal{X}\) is robustly positive invariant for the system \eqref{appenb:sys} and if there exists \(\beta \in \classKL\), \(\lambda \in \classK\) and \(\varphi \in \loro{0}{+\infty}\) such that 
\begin{align}
	\label{e:regional ISS}
    |x(t;\xz,w)| \le \max\aset[\bigg]{\beta\bigl(|\xz|,t\bigr),\lambda\biggl(\sup_{0 \le k \le t}|w_k|\biggr)}+\varphi
\end{align}
	for all \(t \in \Nz\) and for all \(\xz \in \mathcal{X}\), then the system \eqref{appenb:sys} is \emph{regional (practical) input-to-state stable} in \(\mathcal{X}\).
\end{definition}

\begin{theorem}\cite[Theorem 2]{ref:LG:DMR:RS:RISS}
	\label{thrm:appen_b_RISS}
	Let \(\mathcal{X}\) be a robustly positive invariant set for the system \eqref{appenb:sys} and suppose that the system \eqref{appenb:sys} admits a regional ISS-Lyapunov function in \(\mathcal{X}\). Then \eqref{appenb:sys} is regional ISS in \(\mathcal{X}\) in the sense of \eqref{e:regional ISS}.
\end{theorem}